\documentclass{article}

\usepackage{times}
\usepackage{mathptmx}
\usepackage{amsmath,amssymb}

\usepackage{epsfig}
\usepackage{bbm}

\newcommand{\R}{{\mathbbm{R}}} 

\newcommand{\Eb}{{\mathbbm{E}}}

\newcommand{\Nb}{{\mathbbm{N}}}

\renewcommand{\imath}{\textup{i}\hspace{0.05em}}

\DeclareMathOperator{\core}{core}

\DeclareMathOperator{\dist}{dist}
\DeclareMathOperator{\lin}{lin}

\newcommand{\prefeq}{\mathrel{\raisebox{0.035ex}{$\scriptstyle{\unrhd}$}}}
\newcommand{\pref}{\mathrel{\raisebox{0.13ex}{$\scriptstyle{\rhd}$}}}

\newcommand{\notprefeq}{\mathrel{\scriptstyle\ntrianglerighteq}}
\newcommand{\lott}{\mathrel{\raisebox{0.13ex}{$\scriptstyle{\oplus}$}}}

\newcommand{\Z}{\mathcal{Z}}

\newcommand{\F}{\mathcal{F}}
\newcommand{\M}{\mathcal{M}}
\newcommand{\Lc}{\mathcal{L}}
\newcommand{\Pc}{\mathcal{P}}
\newcommand{\B}{\mathcal{B}}

\newcommand{\C}{\mathcal{C}}
\newcommand{\K}{\mathcal{K}}

\newcommand{\Ec}{\mathcal{E}}
\newcommand{\Sc}{\mathcal{S}}
\newcommand{\Q}{\mathcal{Q}}
\newcommand{\Dc}{\mathcal{D}}
\newcommand{\Gc}{\mathcal{G}}
\newcommand{\Jc}{\mathcal{J}}

\newcommand{\1}{\mathbbm{1}}

\newcommand{\Xf}{\mathfrak{X}}

\newcommand{\eqdef}{\mathrel{\overset{\raisebox{-0.02em}{$\scriptstyle\vartriangle$}}{\,=\,}}}
\newcommand{\wto}{\mathrel{\raisebox{-0.2ex}{$\xrightarrow{\,\raisebox{-0.2em}{$\scriptstyle w$}\;}$}}}
\newcommand{\zero}{0}
\newcommand{\ds}{d^{\hspace{-0.04em}\varSigma}\hspace{-0.15em}}
\newcommand{\Dsim}{\mathrel{\overset{\raisebox{-0.02em}{$\scriptstyle\Dc$}}{\sim}}}

\newenvironment{tightlist}[1]{%
    \list{{\textup{(\roman{enumi})}}}{\settowidth\labelwidth{{\textup{(#1)}}}
    \leftmargin\labelwidth \advance\leftmargin\labelsep \itemindent \parindent
    \parsep 0pt plus 1pt minus 1pt \topsep 3pt \itemsep 0pt
    \usecounter{enumi}}}{\endlist}

\newenvironment{tightitemize}{%
    \list{{\textup{$\bullet$}}}{\settowidth\labelwidth{{\textup{}}}
    \leftmargin\labelwidth \advance\leftmargin\labelsep \itemindent \parindent
    \parsep 0pt plus 1pt minus 1pt \topsep 3pt \itemsep 3pt
    }}{\endlist}

\newcommand{\comp}{\circ}
\newcommand{\widebar}{\bar}

\newtheorem{theorem}{Theorem}[section]
\newtheorem{lemma}[theorem]{Lemma}
\newtheorem{definition}[theorem]{Definition}

\newtheorem{corollary}[theorem]{Corollary}
\newtheorem{remark}[theorem]{Remark}

\newenvironment{proof}{\emph{Proof.}}

\title{Common Mathematical Foundations of  Expected Utility\\ 
and  Dual Utility Theories\thanks{To appear in \emph{SIAM Journal on Optimization}}}

\author{Darinka Dentcheva\thanks{Stevens Institute of Technology,
Department of Mathematical Sciences, Castle Point on Hudson, Hoboken, NJ 07030,  Email: {darinka.dentcheva@stevens.edu}.}
\and
Andrzej Ruszczy\'{n}ski\thanks{Rutgers University,
Department of Management Science and Information Systems and RUTCOR,
 94 Rockefeller Rd, Piscataway, NJ 08854, USA, Email: {rusz@business.rutgers.edu}.}}

\date{March 1, 2012}

\begin{document}
\maketitle

\begin{abstract}
We show that the main results of the expected utility  and dual utility theories can be derived in a unified
way from two fundamental mathematical ideas: the separation principle of convex analysis, and integral representations of continuous linear functionals from functional analysis. Our analysis reveals the dual character of utility functions.
We also derive new integral representations of dual utility models.


\emph{AMS}:
Primary: 91B16, 47N10.

\end{abstract}

\pagestyle{myheadings}
\thispagestyle{plain}

\section{Introduction}

The theory of expected utility and the dual utility theory are two very popular and widely accepted approaches for quantification of preferences and a basis of decisions under uncertainty.
These classical topics in economics are covered in plentitude of textbooks and monographs  and represent a benchmark for every other quantitative decision theory.

The expected utility theory of von Neumann and Morgernstern \cite{vNeumann:1944}, and to the dual utility theory of Quiggin \cite{Quiggin:1982} and Yaari \cite{Yaari:1987} are often compared and contrasted (see, \emph{e.g.}, \cite{Fishburn-88}).
Our objective is to show that they have common mathematical roots and their main results can be derived in a unified
way from two mathematical ideas: separation principles of convex analysis, and integral representation theorems for continuous linear functionals.
Our analysis follows similar lines of argument in both cases,  accounting only for the differences of the corresponding prospect spaces.
Our approach reveals the dual nature of both utility functions as continuous linear functionals
on the corresponding prospect spaces. It also elucidates the mathematical limitations of the two approaches and their boundaries.
In addition to this, we obtain new representations of dual utility.


The paper is organized as follows. We briefly review basic concepts of orders and their numerical representation in \S \ref{s:1.1}.
In \S \ref{s:eut}, we focus on the expected utility theory in the prospect space of probability measures on some Polish space of outcomes.
In \S \ref{s:dut}, we derive the dual utility theory in the prospect space of quantile functions. Finally, \S \ref{s:vectors} translates
the earlier results to the prospect spaces of random variables.

\section{Numerical Representation of Preference Relations}
\label{s:1.1}

We start our presentation from the analysis of abstract preference relations in a certain space $\Xf$, which we call the \emph{prospect space}.
We assume that a \emph{preference relation } among prospects is defined by a certain \emph{total preorder}, that is, a binary relation
$\prefeq$ on $\Xf$, which is reflexive, transitive and complete.
The corresponding {indifference relation} $\sim$ is defined in a usual way: $z\sim v$, if $z\prefeq v$ and $v\prefeq z$.
We say that $z$ is \emph{strictly preferred} over $v$ and write it  $z \pref v$, if $z \prefeq v$, and $v \notprefeq z$.

If $\Xf$ is a topological space, we call a preference relation $\prefeq$ \emph{continuous}, if
for every $z\in\Xf$ the sets $\{v\in \Xf: v \prefeq z\}$ and $\{v\in \Xf: z \prefeq v\}$ are closed.

A functional $U:\Xf\to\R$  is a \emph{numerical representation} of the preference relation $\prefeq$ on~$\Xf$, if
\[
z \pref v \iff  U(z) > U(v).
\]
The following classical theorem is the theoretical foundation of the utility theory.
\begin{theorem}
\label{t:Debreu}
Suppose the total preorder $\prefeq$ on a topological space $\Xf$ is continuous and one of the following conditions is satisfied:
\begin{tightlist}{ii}
\item
$\Xf$ is a separable and connected topological space; or
\item
The topology of $\Xf$ has a countable base.
\end{tightlist}
Then there exists a continuous numerical representation of\, $\prefeq$.
\end{theorem}
\begin{remark}
\label{l:gaps}
{\rm
The assertion under (i) is due to \cite[{\S}6]{Eilenberg}. The second case (under (ii)) was announced
in \cite[Thm. II]{Debreu:1954} and corrected in \cite[Thm. 1]{Rader:1963}, but both proofs contained errors.
They were corrected again in \cite{Debreu:1964}; a short and clear proof was eventually provided by \cite{Jaffray}.
For extensions and further discussion, see \cite{BridgesMehta,Mehta:2004}.
}
\end{remark}

This is the starting point of our considerations. The expected utility theory and the dual utility theory derive properties of the numerical
representation $U(\cdot)$ and its integral representations in specific prospect spaces and under additional conditions on the pre\-order~$\prefeq$.
These conditions are associated with the operation of forming convex combinations of prospects.
In the expected utility theory, the prospects are probability distributions,
and their convex combinations correspond to lotteries.
The dual utility theory uses convex combinations of comonotonic real random variables, which translates to forming
convex combinations of quantile functions.

It is evident that convexity in some underlying vector space is a key property in the system of axioms of the expected utility
and dual utility models.
Both theories have been developed using different mathematical approaches and specialized tools.
Our objective is to show that they can be deduced in a unified way
from the fundamental separation theorem of convex analysis,
and from functional analysis results about integral representation of continuous linear functionals in topological vector spaces.

The foundation of our approach is the separation principle for convex sets having non\-empty algebraic interiors.
 The \emph{algebraic interior} of a convex set $A$ in a vector space $\mathcal{Y}$ is defined as follows:
\[
\core(A) = \big\{x\in A: \forall\,(d\in \mathcal{Y})\;\exists\,(t>0)\ x+t d \in A\big\}.
\]
The following separation theorem is due to \cite{Eidelheit:1936} and \cite{Dieudonne:1941}; see also \cite{Klee:1951}.
\begin{theorem}
\label{t:sepcore}
Suppose $\mathcal{Y}$ is a vector space and $A\subset \mathcal{Y}$ is a convex set. If $\core(A)\ne \emptyset$ and
$x\not\in \core(A)$, then there exists a linear functional $\ell$ on $\mathcal{Y}$
such that $\ell(x) < \ell(y)$ for all $y\in \core(A)$.
\end{theorem}

In the development of the expected utility theory of von Neumann and Morgenstern in \S 3, and of
the dual utility theory of Yaari and Quiggin in \S \ref{s:dut}, we apply the same method:
\begin{tightitemize}
\item
Embedding of the prospect space into an appropriate vector space;
\item
Representation of the set of pairs of comparable prospects by a convex set with a nonempty algebraic interior;
\item
Application of the separation theorem to establish the existence of an affine numerical representation;
\item
Application of an appropriate integral representation theorem for continuous linear functionals to derive the existence of utility
and dual utility functions.
\end{tightitemize}

\section{Expected Utility Theory}
\label{s:eut}

\subsection{The Prospect Space of Distributions}
\label{s:prospect-lotteries}

Given  a Polish space $\Sc$, equipped with its $\sigma$-algebra $\B$ of Borel sets, we consider the set $\Pc(\Sc)$ of probability measures on $\Sc$.
The theory of expected utility can be formulated in a rather general way for the prospect space $\Xf=\Pc(\Sc)$.

We assume that the preference relation $\prefeq$ satisfies two additional conditions:
\begin{description}
\item[\emph{Independence Axiom:}] For all $\mu$, $\nu$, and $\lambda$ in  $\Pc(\Sc)$ one has
\[
\mu \pref \nu \;\Longrightarrow\;
\alpha \mu + (1-\alpha) \lambda \pref \alpha \nu + (1-\alpha) \lambda,\quad \forall\, \alpha\in (0,1),
\]
\item[\emph{Archimedean Axiom:}] For all $\mu$, $\nu$, and $\lambda$ in  $\Pc(\Sc)$, satisfying the relations
$\mu \pref \nu \pref \lambda$,
there exist $\alpha,\beta\in (0,1)$ such that
\[
\alpha \mu + (1-\alpha) \lambda \pref \nu \pref \beta \mu + (1-\beta) \lambda.
\]
\end{description}
These are exactly the conditions assumed in the pioneering work \cite{vNeumann:1944}
(see also \cite[{\S}8.2,{\S}8.3]{Fishburn70}, \cite[{\S}2.2]{Fishburn82}, \cite[{\S}2.2]{FollmerSchied}, \cite{HersteinMilnor}).

Our idea is to exploit convexity in a more transparent fashion.
We derive the following properties of
a preorder satisfying the independence and Archimede\-an axioms.
\begin{lemma}
\label{l:indiff}
Suppose a total preorder $\prefeq$ on $\Pc(\Sc)$ satisfies
the independence axiom. Then for every $\mu\in \Pc(\Sc)$ the indifference set $\{\nu\in \Pc(\Sc): \nu \sim \mu\}$ is convex.
\end{lemma}
\begin{proof}
Let $\nu\sim\mu$ and $\lambda\sim\mu$.  Suppose $(1-\alpha) \nu + \alpha \lambda \pref \nu$ for some $\alpha\in (0,1)$. Then also
 $(1-\alpha) \nu + \alpha \lambda \pref \lambda$.
Using the independence axiom with these two relations, we obtain contradiction as follows:
 \begin{align*}
(1-\alpha) \nu + \alpha \lambda &= (1-\alpha)\big[ (1-\alpha) \nu + \alpha \lambda\big] + \alpha \big[ (1-\alpha) \nu + \alpha \lambda\big]\\
&\pref (1-\alpha)\nu + \alpha \big[ (1-\alpha) \nu + \alpha \lambda\big]
\pref (1-\alpha)\nu + \alpha \lambda.
 \end{align*}
The case when $\nu \pref (1-\alpha) \nu + \alpha \lambda$ is excluded in a similar way. We conclude that  $(1-\alpha) \nu + \alpha \lambda \sim \mu$,
for all $\alpha\in (0,1)$.
\end{proof}
\begin{remark}
\label{r:indiff}
{\rm
Lemma \ref{l:indiff} derives the properties of \emph{quasi-concavity}
and \emph{quasi-convexity},
that is, \emph{quasi-linearity} of the preorder~$\prefeq$ (see, \emph{e.g.},
\cite[{\S} 9.2]{Quiggin:1993}, and the references therein). The property of quasi-concavity
is called \emph{uncertainty aversion} in \cite{GilboaSchmeidler,Schmeidler:1989}.
}
\end{remark}
\begin{lemma}
\label{l:arch2}
Suppose a total preorder $\prefeq$ on $\Pc(\Sc)$ satisfies
the independence and Archi\-me\-de\-an axioms. Then for all $\mu,\nu\in \Pc(\Sc)$, satisfying the relation
$\mu \pref \nu$, and for all $\lambda \in \Pc(\Sc)$, there exists $\bar{\alpha}>0$ such that
\begin{equation}
\label{arch2}
(1-\alpha) \mu + \alpha \lambda \pref \nu \quad \text{and} \quad
\mu \pref (1-\alpha) \nu + \alpha \lambda,\quad \forall\,\alpha\in [0,\bar{\alpha}].
\end{equation}
\end{lemma}
\begin{proof}
We focus on the left relation in \eqref{arch2} and consider three cases.

\noindent
\emph{Case 1:} $\nu\pref\lambda$. The left relation in \eqref{arch2} is true for some $\bar{\alpha}_1\in (0,1)$,
owing to the Archimedean axiom.
If $\alpha\in [0,\bar{\alpha}_1]$ then for $\beta=\alpha/\bar{\alpha}_1\in [0,1]$  the independence axiom yields
\[
(1-\alpha) \mu + \alpha \lambda  =
(1-\beta) \mu   + \beta  \big[(1-\bar{\alpha}_1) \mu + \bar{\alpha}_1 \lambda \big] \pref (1-\beta) \mu + \beta \nu \pref \nu.
\]

\noindent
\emph{Case 2:} $\lambda \pref \nu$. Applying the independence axiom twice, we obtain
\[
(1-\alpha) \mu + \alpha \lambda \pref (1-\alpha) \nu + \alpha \lambda \pref \nu,\quad
\forall\,\alpha \in [0,1).
\]
\noindent
\emph{Case 3:} $\lambda \sim \nu$. By virtue of Lemma \ref{l:indiff},  $(1-\alpha) \nu + \alpha \lambda \sim \nu$ for all $\alpha\in [0,1)$, and
the left relation in \eqref{arch2} follows  from the in\-de\-pendence axiom.

This proves the left relation in \eqref{arch2} for all $\alpha \in [0,\bar{\alpha}_1]$ with some $\bar{\alpha}_1 >0$.

Reversing the preference relation, that is, defining $\nu \pref_{\!_{-1}} \mu \iff \mu \pref \nu$, the right relation in \eqref{arch2} follows analogously. We infer the existence of  some $\bar{\alpha}_2 >0$, such that
the right relation in \eqref{arch2} is true for all $\alpha \in [0,\bar{\alpha}_2]$.
Setting $\bar{\alpha} = \min\{\bar{\alpha}_1,\bar{\alpha}_2\}$ we obtain the assertion of the lemma.
\end{proof}

\subsection{Affine Numerical Representation}
\label{s:affine-utility}

The set $\Pc(\Sc)$ is a convex subset of the vector space $\M(\Sc)$ of signed regular finite measures on $\Sc$.
It is also convenient for our derivations to consider the linear subspace $\M_0(\Sc) \subset \M(\Sc)$ of signed regular measures $\mu$
such that $\mu(\Sc)=0$.

The main theorem of this section is due to \cite{vNeumann:1944}. Its complicated constructive proof has been
since reproduced in many sources (see, e.g.,  \cite[Thm. 2.21]{FollmerSchied} and the references therein),
or emulated in the setting of  \emph{mixture sets}  (see, e.g., \cite[Thm. 8.4]{Fishburn70}, \cite[Thm. 2, Ch. 2]{Fishburn82},
\cite{HersteinMilnor}, and the references therein).
Our proof, as indicated in the introduction, is based on the separation theorem.
\begin{theorem}
\label{t:vNM1}
Suppose the total preorder $\prefeq$ on $\Pc(\Sc)$ satisfies the independence and Archi\-me\-dean axioms. Then there exists a linear functional
on $\M(\Sc)$, whose restriction to $\Pc(\Sc)$ is a numerical representation of $\,\prefeq$.
\end{theorem}
\begin{proof}
In the space $\M_0(\Sc)$, define the set
\[
C_0 = \{\mu -\nu: \mu\in \Pc(\Sc),\, \nu\in \Pc(\Sc),\, \mu \pref \nu\}.
\]
Consider two arbitrary points $\vartheta$ and $\varkappa$ in $C_0$, that is,
\begin{align*}
\vartheta = \mu-\nu,\qquad \mu,\nu\in \Pc(\Sc),\qquad \mu \pref \nu, \\
\varkappa = \lambda-\sigma,\qquad \lambda,\sigma\in \Pc(\Sc),\qquad \lambda \pref \sigma.\notag
\end{align*}
 For every  $\alpha \in (0,1)$, using the independence axiom twice, we obtain
\[
\alpha \mu + (1-\alpha) \lambda \pref \alpha \nu + (1-\alpha) \lambda \pref \alpha \nu + (1-\alpha) \sigma.
\]
Therefore, $\alpha\vartheta + (1-\alpha) \varkappa \in C_0$, which proves that $C_0$ is convex.

Define $C = \{\alpha \vartheta: \vartheta\in C_0,\, \alpha>0 \}$. It is evident that $C$ is convex cone, that is, for all $\vartheta,\varkappa\in C$, and
all $\alpha,\beta>0$ we have $\alpha \vartheta + \beta \varkappa \in C$. Moreover, $C\subset \M_0$.

We shall prove that the algebraic interior of $C$ is nonempty, and that, in fact, $C=\core(C)$. Consider any $\vartheta\in C$,
an arbitrary nonzero measure $\lambda\in \M_0$, and the ray
\[
z(\tau) = \vartheta + \tau \lambda,\quad \tau>0.
\]
Our objective is to show that $z(\tau)\in C$ for a sufficiently small $\tau>0$. Let $\lambda=\lambda^+ - \lambda^-$
be the Jordan decomposition of $\lambda$. With no loss of generality, we may assume that the direction $\lambda$ is normalized
so that $|\lambda| = \lambda^+(\Sc)+ \lambda^-(\Sc) =2$. As $\lambda\in \M_0$, we have then $\lambda^+(\Sc)=\lambda^-(\Sc) = 1$.
Let $\alpha>0$ be such that the point $\vartheta_0=\alpha \vartheta \in C_0$. Since $C$ is a cone,
$z(\tau)\in C$ if and only if $\alpha z(\tau)\in C$. Setting $t=\alpha \tau$, we reformulate our question as follows:
Does $\vartheta_0 + t \lambda$ belong to $C$ for sufficiently small $t>0$?
Since $\vartheta_0\in C_0$, we can represent it as a difference $\vartheta_0=\mu-\nu$, with $\mu,\nu\in \Pc(\Sc)$, and $\mu \pref \nu$. Then
\begin{equation}
\label{core1}
\vartheta_0 + t \lambda = \big[ (1-t)\mu + t \lambda^+\big] - \big[ (1-t)\nu + t \lambda^-\big] + t \vartheta_0.
\end{equation}
Both expressions in brackets are probability measures for $t\in [0,1]$. By virtue of the independence axiom,
\[
\mu \pref \frac{1}{2}\mu + \frac{1}{2}\nu \pref \nu.
\]
By Lemma \ref{l:arch2}, there exists $t_0>0$, such that for all $t\in [0,t_0]$ we also have
\[
(1-t)\mu + t \lambda^+ \pref \frac{1}{2}\mu + \frac{1}{2}\nu \pref (1-t)\nu + t \lambda^-.
\]
This proves that
\[
\big[ (1-t)\mu + t \lambda^+\big] - \big[ (1-t)\nu + t \lambda^-\big] \in C_0,
\]
provided that $t\in [0,t_0]$. For these values of $t$, the right hand side of \eqref{core1}
is a sum of two elements of~$C$.
As the set $C$ is a convex cone, this sum is an element of $C$ as well.
Consequently, $\vartheta + \tau \lambda\in C$ for all $\tau\in [0,t_0/\alpha]$.

Summing up, $C$ is convex, $C=\core(C)$, and $0\notin C$. By Theorem \ref{t:sepcore}, the point 0 and the set $C$
can be separated strictly:
there exists a linear functional $U_0$ on $\M_0(\Sc)$,
such that
\begin{equation}
\label{U0}
U_0(\vartheta) > 0,\quad \forall\, \vartheta \in C.
\end{equation}
We can extend the linear functional $U_0$ to the whole space $\M(\Sc)$ by choosing a measure $\lambda \in \Pc(\Sc)$ and setting
\[
U(\mu) = U_0\big(\mu - \mu(\Sc)\lambda\big),\quad \mu \in \M(\Sc).
\]
It is linear and coincides with $U_0$ on $\M_0(\Sc)$. Relation (\ref{U0}) is equivalent to the following statement: for all
$ \mu,\nu \in \Pc(\Sc)$ such that $\mu \pref \nu$, we have
\[
U_0(\mu-\nu) = U(\mu-\nu) = U(\mu) - U(\nu) > 0.
\]
It follows that $U$ restricted to $\Pc(\Sc)$ is the postulated affine numerical representation of the pre\-order~$\prefeq$.
\end{proof}

\subsection{Integral Representation. Utility Functions}
\label{s:integral-utility}

To prove the main result of this section, we assume that the space $\M(\Sc)$ is equipped with the topology of weak convergence of
measures. Recall that a sequence of measures
$\{\mu_n\}$ converges weakly to $\mu$ in $\M(\Sc)$, which we write $\mu_n\wto \mu$,  if
\[
\lim_{n\to\infty} \int_{\Sc} f(z)\, \mu_n(dz) = \int_{\Sc} f(z)\, \mu(dz), \quad\forall\, f\in\C_{\textup{b}}(\Sc),
\]
where $\C_{\textup{b}}(\Sc)$ is the set of bounded continuous real functions on $\Sc$ (for more details see,
e.g.,~\cite{Billingsley:1999}).

We derive our next result from the classical Banach's theorem on weakly$^\star$ continuous functionals.
It has been proved in the past via discrete approximations of the measures in question (see, \emph{e.g.},
\cite[{\S}10]{Fishburn70}, \cite[Ch. 3, Thm. 1--4]{Fishburn82} and \cite[Thm. 2.28]{FollmerSchied}).
\begin{theorem}
\label{t:vNM2}
Suppose the total preorder $\prefeq$ on $\Pc(\Sc)$ is continuous and satisfies the independence axiom. Then a continuous and
bounded function $u:\Sc\to \R$ exists, such that the functional
\begin{equation}
\label{vNMrep}
U(\mu) = \int_{\Sc}u(z)\;\mu(dz)
\end{equation}
is a numerical representation of $\,\prefeq$ on $\Pc(\Sc)$.
\end{theorem}
\begin{proof}
The continuity of the preorder $\prefeq$ implies the Archimedean axiom. Indeed, the sets $\{ \pi\in \Pc(\Sc) : \pi \pref \nu\}$
and $\{ \pi\in \Pc(\Sc) : \mu \pref \pi\}$ are
open, and the mapping $\alpha\mapsto \alpha\pi + (1-\alpha) \lambda$, $\alpha\in [0,1]$, is continuous for any $\lambda \in \Pc(\Sc)$.

Owing to Theorem \ref{t:vNM1},
a linear functional $U:\M(\Sc)\to \R$ exists, whose restriction to $\Pc(\Sc)$ is a numerical representation of~$\prefeq$.
We shall prove that the functional $U(\cdot)$ is continuous on $\Pc(\Sc)$, that is, for every $\alpha$ the sets
\[
A = \{\mu\in \Pc: U(\mu) \le \alpha\}\qquad \text{and} \qquad  B = \{\mu\in \Pc: U(\mu) \ge \alpha\}
\]
are closed. Since $\Pc$ is convex and $U(\cdot)$ is linear, the set
$U(\Pc)$ is convex. Therefore, for every  $\alpha$ one of three cases may occur:
\begin{tightlist}{iii}
\item $U(\mu) < \alpha$ for all $\mu \in \Pc$;
\item $U(\mu) > \alpha$ for all $\mu \in \Pc$;
\item $\alpha \in U(\Pc)$.
\end{tightlist}
In cases (i) and (ii) there is nothing to prove. In case (iii), let $\nu\in \Pc$ be such that $U(\nu)=\alpha$. Since
$U(\cdot)$ is a numerical representation of the preorder, we have
\[
A = \{\mu\in \Pc: \nu \prefeq \mu \}\qquad \text{and} \qquad  B = \{\mu\in \Pc: \mu \prefeq \nu\}.
\]
Both sets are closed due to the continuity of the preorder $\prefeq$.

Now, we can prove continuity on the whole space $\M(\Sc)$.
Suppose  $\mu_n\wto\mu$, but $U(\mu_n)$ does not converge to $U(\mu)$. Then an infinite set $\K$ and $\varepsilon>0$ exist such that
$|U(\mu_k)-U(\mu)| > \varepsilon$ for all $k\in \K$.
As $U(\cdot)$ is linear, with no loss of generality we may assume that $\mu\in \Pc$.
Consider the Jordan decomposition $\mu_k=\mu^+_k-\mu^-_k$. By the Prohorov theorem \cite{Prohorov:1956}, 
the sequence $\{\mu_k\}$ is uniformly tight,
and so are $\{\mu^+_k\}$ and $\{\mu^-_k\}$. They are, therefore, weakly compact. Let $\nu$ be the weak limit of a convergent subsequence
$\{\mu^+_k\}_{k\in \K_1}$, where $\K_1\subseteq \K$.
 Then the subsequence $\{\mu^-_k\}_{k\in \K_1}$ also has a weak limit: $\lambda = \nu-\mu$. The measures
$\mu^+_k/\mu^+_k(\Sc)$ are probability measures, and $\mu^+_k(\Sc)\to \nu(\Sc) \ge 1$. Consequently,
\[
U\big(\mu^+_k\big) = \mu^+_k(\Sc) U\bigg(\frac{\mu^+_k}{\mu^+_k(\Sc)}\bigg) \xrightarrow{k\in \K_1} \nu(\Sc) U\bigg(\frac{\nu}{\nu(\Sc)}\bigg) = U(\nu).
\]
Similarly, $\mu^-_k(\Sc)\xrightarrow{k\in \K_1} \nu(\Sc)$ and
\[
U\big(\mu^-_k\big) = \mu^-_k(\Sc) U\bigg(\frac{\mu^-_k}{\mu^-_k(\Sc)}\bigg),\quad \text{if} \quad \mu^-_k(\Sc)> 0.
\]
If $\mu^-_k(\Sc)> 0$ infinitely often, then
the limit of $U\big(\mu^-_k\big)$ on this sub-subsequence equals $U(\lambda)$. If $\mu^-_k=0$ infinitely often, then $\lambda=0$.
In any case, $U(\mu^-_k)\to U(\lambda)$, when $k\in \K_1$. It follows that
\[
U(\mu_k) = U(\mu^+_k) - U(\mu^-_k) \xrightarrow{k\in \K_1} U(\nu)- U(\lambda) = U(\mu),
\]
 which contradicts
our assumption. Therefore, the functional $U(\cdot)$ is continuous on $\M(\Sc)$.
Owing to Theorem \ref{t:Mdual} in the Appendix, $U(\cdot)$ has the form \eqref{vNMrep}, where $u:\Sc\to \R$ is continuous and bounded.
\end{proof}

Formula \eqref{vNMrep} is referred to as the \emph{expected utility representation}, and $u(\cdot)$ is called the \emph{utility function}. 

The utility function in Theorem \ref{t:vNM2} is bounded. If we restrict the space of measures to measures satisfying additional integrability
conditions, we obtain representations in which unbounded utility functions may occur. Our construction is similar to the construction leading to \cite[Thm. 2.30]{FollmerSchied} with the difference that we work with the space of signed measures on $\Sc$, rather than with the set of probability measures.

Let $\psi:\Sc\to [1,\infty)$ be a continuous \emph{gauge function}, and let $\C_{\textup{b}}^\psi(\Sc)$ be the set of functions
$f:\Sc\to \R$, such that $f/\psi\in \C_{\textup{b}}(\Sc)$. We can define the space $\M^{\psi}(\Sc)$ of regular signed measures $\mu$, such that
\[
\Big| \int_{\Sc} f(z)\;\mu(dz) \Big| < \infty,\quad \forall\, f\in \C_{\textup{b}}^\psi(\Sc).
\]
Similarly to the topology of weak convergence, we say that
a sequence of measures $\mu_n\in \M(\Sc)$ is convergent $\psi$-weakly to $\mu\in \M(\Sc)$ if
\[
\lim_{n\to\infty} \int_{\Sc} f(z)\, \mu_n(dz) = \int_{\Sc} f(z)\, \mu(dz), \quad \forall\, f\in \C_{\textup{b}}^\psi(\Sc).
\]
All continuity statements will be now made with respect to this topology.
We use the symbol $\Pc^\psi(\Sc)$ to denote the set of probability measures
in $\M^{\psi}(\Sc)$.

We can now recover the result of \cite[Th. 2.30]{FollmerSchied}.
\begin{theorem}
\label{t:vNM2psi}
Suppose the total preorder $\prefeq$ on $\Pc^\psi(\Sc)$ is continuous and satisfies the independence axiom. Then a function
$u \in \C_{\textup{b}}^\psi(\Sc)$ exists such that the functional
\begin{equation}
\label{vNMreppsi}
U(\mu) = \int_{\Sc}u(z)\;\mu(dz)
\end{equation}
is a numerical representation of $\,\prefeq$ on $\Pc^\psi(\Sc)$.
\end{theorem}
\begin{proof}
The proof is identical to the proof of Theorem \ref{t:vNM2}, except that we need to invoke Theorem \ref{t:Mdualpsi}
from the Appendix.
\end{proof}

\subsection{Monotonicity and Risk Aversion}
\label{s:ra-mu}

Suppose $\Sc$ is a separable Banach lattice with a partial order relation $\ge$.
In a lattice structure, it makes sense to speak about monotonicity of a preference relation. In this section,
the symbol $\delta_z$ denotes a unit atomic measure concentrated on $z\in \Sc$.

\begin{definition}
\label{d:monotonic}
A preorder $\prefeq$ on $\Pc(\Sc)$ is \emph{monotonic} with respect to the partial ord\-er~$\ge$
on $\Sc$, if for all $z,v\in \Sc$ the  implication
$z \ge v \;\Longrightarrow\; \delta_z  \prefeq \delta_v$ is true.
\end{definition}

We can derive monotonicity of utility functions from the monotonicity of the order.
\begin{theorem}
\label{t:vNM-mono}
Suppose the total preorder $\prefeq$ on $\Pc(\Sc)$ is monotonic, continuous, and satisfies the independence axiom. Then a nondecreasing, continuous, and
bounded function $u:\Sc\to \R$ exists, such that the functional \eqref{vNMrep}
is a numerical representation of $\,\prefeq$ on $\Pc(\Sc)$.
\end{theorem}
\begin{proof}
In view of Theorem \ref{t:vNM2}, it is sufficient to verify that the function $u(\cdot)$ in \eqref{vNMrep}
is nodecreasing with respect to the partial order $\ge$.
To this end, we consider $z,v\in\Sc$ such that $z\ge v$. By monotonicity of the order,
$u(z) = U(\delta_z) \ge U(\delta_v) = u(v)$.
\end{proof}

We now focus on the case, when the gauge function is $\psi_p(z)=1+\|z\|^p$, where $p\ge 1$.
Then for every $\mu\in \Pc^{\psi_p}(\Sc)$ and for every $\sigma$-subalgebra $\Gc$ of $\B$ the conditional expectation
$\Ec_{\mu|\Gc}:\Sc\to\Sc$ is well-defined, as a $\Gc$-measurable function satisfying the equation
\[
\int_G \Ec_{\mu|\Gc}(z)\;\mu(dz) = \int_G z\;\mu(dz),\quad G\in\Gc
\]
 (\emph{cf.} \cite[{\S}2.1]{LedouxTalagrand}).
The conditional expectation $\Ec_{\mu|\Gc}$ induces a probability measure on $(\Sc,\B)$ as follows
\[
\mu_{\Gc}(A) = \mu\big\{\Ec_{\mu|\Gc}^{-1}(A)\big\},\quad A \in \B.
\]
\begin{definition}
\label{d:risk-averse-mu}
A preference relation $\prefeq$ on $\Pc^{\psi_p}(\Sc)$ is \emph{risk-averse}, if
$ \mu_{\Gc}\prefeq \mu$,
for every $\mu\in \Pc^{\psi_p}(\Sc)$ and every $\sigma$-subalgebra $\Gc$ of $\B$.
\end{definition}

By choosing $\Gc=\{\Sc,\emptyset\}$, we observe that Definition \ref{d:risk-averse-mu} implies that
$\delta_{\Ec_{\mu}} \prefeq \mu$, where $\Ec_\mu=\int_{\Sc}z\;\mu(dz)$ is the expected value.
\begin{theorem}
\label{t:vNM2psi-ra}
Suppose a total preorder $\prefeq$ on $\Pc^{\psi_p}(\Sc)$ is continuous, risk-averse, and satisfies the independence axiom. Then a concave function
$u \in \C_{\textup{b}}^{\psi_p}(\Sc)$ exists such that the functional \eqref{vNMreppsi}
is a numerical representation of $\,\prefeq$ on $\Pc^\psi(\Sc)$.
\end{theorem}
\begin{proof} In view of Theorem \ref{t:vNM2psi}, we only need to prove the concavity of $u(\cdot)$. Due to risk aversion,
for every $\mu\in\Pc^{\psi_p}(\Sc)$, we obtain $\delta_{\Ec_{\mu}} \prefeq \mu$. Consequently,
\[
u\bigg(\int_{\Sc}z\;\mu(dz)\bigg) \ge  \int_{\Sc}u(z)\;\mu(dz).
\]
This is Jensen's inequality, which is equivalent to the concavity of $u(\cdot)$.
\end{proof}
\begin{remark}
\label{r:eisk-aversion}
{\rm
It is clear from the proof that the concavity of $u(\cdot)$ could have been obtained by simply assuming that $\delta_{\Ec_{\mu}} \prefeq \mu$.
The concavity of $u(\cdot)$ would imply risk aversion in the sense of Definition \ref{d:risk-averse-mu}, by virtue of Jensen's inequality
for conditional expectations. Therefore, Definition \ref{d:risk-averse-mu} and the requirement that $\delta_{\Ec_{\mu}} \prefeq \mu$ are  equivalent
within the framework of the expected utility theory. Nonetheless, we prefer to leave Definition \ref{d:risk-averse-mu}
in its full form, because we shall use the concept of risk aversion in connection with other axioms, where such equivalence cannot be derived.
}
\end{remark}

\section{Dual Utility Theory}
\label{s:dut}
\subsection{The Prospect Space of Quantile Functions}
\label{s:dual-prospects}
The dual utility theory is formulated in much more restrictive setting: for the probability distributions on the real line. With every
probability distribution $\mu\in \Pc(\R)$ we associate the distribution function: $F_{\mu}(t)  \eqdef \mu\big((-\infty,t]\big)$. It is
nondecreasing and right-continuous. We can, therefore, define its inverse
\begin{equation}
\label{quantile-function}
F^{-1}_\mu(p) \eqdef \inf\,\{ t\in \R : F_{\mu}(t) \ge p\},\quad p\in (0,1).
\end{equation}
By definition, $F^{-1}_\mu(p)$ is the smallest $p$-quantile of $\mu$. We call $F^{-1}_\mu(\cdot)$ the \emph{quantile function} associated with the
probability measure $\mu$.
Every quantile function is nondecreasing and left-continuous on the open interval $(0,1)$. On the other hand, every nondecreasing and left-continuous
function $\varPhi(\cdot)$ on $(0,1)$ uniquely defines  the following distribution function:
\[
F_\mu(t) = \varPhi^{-1}(t) \eqdef \sup\, \{ p\in (0 ,1) : \varPhi(p) \le t \},
\]
which corresponds to a certain probability measure $\mu \in \Pc(\R)$.

\emph{The set $\Q$ of all nondecreasing and left-continuous functions on the interval $(0,1)$ will be our prospect space.}
It is evident that
$\Q$ is a convex cone in the vector space $\Lc_0(0,1)$ of all Lebesgue measurable functions on the interval $(0,1)$.

We assume that the preference relation $\prefeq$ on $\Q$ is a total preorder  and satisfies two additional conditions:
\begin{description}
\item[\emph{Dual Independence Axiom:}] For all $\varPhi$, $\varPsi$, and $\varUpsilon$ in  $\Q$ one has
\[
\varPhi \pref \varPsi \;\Longrightarrow\;
\alpha \varPhi + (1-\alpha) \varUpsilon \pref \alpha \varPsi + (1-\alpha) \varUpsilon,\quad \forall\, \alpha\in (0,1),
\]
\item[\emph{Dual Archimedean Axiom:}] For all $\varPhi$, $\varPsi$, and $\varUpsilon$ in  $\Q$, satisfying the relations
$\varPhi \pref \varPsi \pref \varUpsilon$,
there exist $\alpha,\beta\in (0,1)$ such that
\[
\alpha \varPhi + (1-\alpha) \varUpsilon \pref \varPsi \pref \beta \varPhi + (1-\beta) \varUpsilon.
\]
\end{description}

In \cite{Yaari:1987},
the dual utility theory considered the space of uniformly bounded random variables
on an implicitly assumed atomless probability
space. The operation of forming convex combinations was considered for comonotonic random variables only. This corresponds
to forming convex combinations of quantile functions, and in this way our system of axioms is a subset of the
axioms of the dual utility theory. We discuss this issue in \S \ref{s:dual-utility-variables}.

Similarly to Lemmas \ref{l:indiff} and  \ref{l:arch2}, we derive the following properties of a preorder satisfying the dual axioms.
\begin{lemma}
\label{l:indiff-dual}
Suppose a total preorder $\prefeq$ on $\Q$ satisfies
the dual independence axiom. Then for every $\varPhi\in \Q$ the indifference set $\{\varPsi\in \Q: \varPsi \sim \varPhi\}$ is convex.
\end{lemma}

\begin{lemma}
\label{l:arch2q}
Suppose a total preorder $\prefeq$ on $\Q$ satisfies
the dual independence and Archimedean axioms. Then for all $\varPhi,\varPsi\in \Q$, satisfying the relation
$\varPhi \pref \varPsi$, and for all $\varUpsilon \in \Q$, there exists $\bar{\alpha}>0$ such that
\begin{equation}
\label{arch2q}
(1-\alpha) \varPhi + \alpha \varUpsilon \pref \varPsi \quad \text{and} \quad
\varPhi \pref (1-\alpha) \varPsi + \alpha \varUpsilon, \quad \forall\,\alpha\in [0,\bar{\alpha}].
\end{equation}
\end{lemma}

\subsection{Affine Numerical Representation}

This section corresponds to \$ \ref{s:affine-utility} and it contains the proof of existence of an affine utility functional
representing a total preorder, which satisfies the dual independence and Archi\-me\-dean axioms. To the best of our knowledge, this result is new in its formulation
and derivation.

It is convenient for our derivations to consider the linear span of $\Q$ defined as follows:
\[
\lin(\Q)=\bigg\{\sum_{i=1}^k\alpha_i\varPhi_i: \alpha_i\in\R,\,\varPhi_i\in\Q,\, i=1,\dots,k,\,k\in \Nb\bigg\} = \Q-\Q,
\]
where $\Q-\Q$ is the Minkowski sum of the sets $\Q$ and $-\Q$.
The relation follows from the fact that $\Q$ is a convex cone.
\begin{theorem}
\label{t:affine-quantile}
If a total preorder $\prefeq$ on $\Q$  satisfies the dual independence and Archi\-me\-dean axioms, then  a linear functional on $\lin(\Q)$ exists, whose restriction to $\Q$ is a numerical representation of~$\,\prefeq$.
\end{theorem}
\begin{proof}
Define in the space $\lin(\Q)$ the set
\[
C = \{\varPhi -\varPsi: \varPhi\in \Q,\, \varPsi\in \Q,\, \varPhi \pref \varPsi\}.
\]
Exactly as in the proof of Theorem \ref{t:vNM1},
we can prove that $C$ is convex. We shall prove that it is a  cone. Suppose
$\varPhi \pref \varPsi$ and let $\alpha > 0$. If $\alpha\in(0,1)$, then the independence axiom implies that
\[
\alpha\varPhi = \alpha\varPhi + (1-\alpha)\zero \pref \alpha\varPsi + (1-\alpha)\zero = \alpha\varPsi.
\]
Consider $\alpha>1$, and suppose $\alpha \varPsi \prefeq \alpha \varPhi$. If $\alpha \varPsi \pref \alpha \varPhi$,
then, owing to the independence axiom, we obtain a contradiction:
$\varPsi=\frac{1}{\alpha}(\alpha\varPsi)\pref \frac{1}{\alpha}(\alpha\varPhi) = \varPhi$. Consider the case when $\alpha \varPsi \sim \alpha \varPhi$.
By virtue of Lemma \ref{l:indiff-dual} and the independence axiom, for any $\beta\in (0,1/\alpha)$ we obtain a contradiction in the following way:
\begin{align*}
\alpha \varPsi &\sim \beta (\alpha \varPhi) + (1-\beta)(\alpha\Psi)   = (\beta\alpha)\varPhi + (1-\beta\alpha)\Big[ \frac{(1-\beta)\alpha}{1-\beta\alpha}\varPsi\Big] \\
& \pref (\beta\alpha)\varPsi + (1-\beta)(\alpha\Psi)  = \alpha\varPsi.
\end{align*}
Therefore, $\alpha\varPhi \pref \alpha\varPsi$ for all $\alpha>0$.
We conclude that
for every $\alpha>0$ the element $\alpha(\varPhi-\varPsi)\in C$. Consequently, $C$ is a convex cone.

To prove that the algebraic interior of $C$ is nonempty, and that in fact $C=\core(C)$, we repeat the argument
 from the proof of Theorem \ref{t:vNM1}. Consider any $\varGamma\in C$,
a  function $\varUpsilon\in \lin(\Q)$, and the ray
$Z(\tau) = \varGamma + \tau \varUpsilon$, where $ \tau>0$. By the definition of $\lin(\Q)$,
we can represent  $\varUpsilon=\varUpsilon^+ - \varUpsilon^-$, with $\varUpsilon^+ , \varUpsilon^-\in \Q$.

Since $\varGamma\in C$, we can represent it as a difference $\varGamma=\varPhi-\varPsi$, with $\varPhi,\varPsi\in \Q$,
and $\varPhi \pref \varPsi$. Then
\begin{equation}
\label{core1q}
\varGamma + t \varUpsilon = \big[ (1-t)\varPhi + t \varUpsilon^+\big] - \big[ (1-t)\varPsi + t \varUpsilon^-\big] + t \varGamma.
\end{equation}
Both expressions in brackets are elements of $\Q$. By the dual independence axiom,
\[
\varPhi \pref \frac{1}{2}\varPhi + \frac{1}{2}\varPsi \pref \varPsi.
\]
By Lemma \ref{l:arch2q}, there exists $t_0>0$, such that for all $t\in [0,t_0]$ we also have
\[
(1-t)\varPhi + t \varUpsilon^+ \pref \frac{1}{2}\varPhi + \frac{1}{2}\varPsi \pref (1-t)\varPsi + t \varUpsilon^-.
\]
This proves that
\[
\big[ (1-t)\varPhi + t \varUpsilon^+\big] - \big[ (1-t)\varPsi + t \varUpsilon^-\big] \in C,
\]
provided that $t\in [0,t_0]$. Thus relation \eqref{core1q} implies that for every $t\in [0,t_0]$
the point $\varGamma + t \varUpsilon$  is a sum of two elements of $C$.
Since the set $C$ is a convex cone, this point is also an element of $C$.

As $C$ is convex, $C=\core(C)$, and $0\notin C$, the  assumptions of Theorem \ref{t:sepcore} are satisfied. Therefore, $\zero$ and $C$ can be separated strictly:
there exists a linear functional $U$ on $\lin(\Q)$,
such that
$U(\varGamma) > 0$, for all $\varGamma \in C$.
Thus,
\[
U(\varPhi) - U(\varPsi) > 0,\quad \text{if}\quad \varPhi \pref \varPsi,
\]
as required.
\end{proof}

\subsection{Integral Representation with Rank Dependent Utility Functions}

In order to derive an integral representation of the numerical representation $U(\cdot)$ of the preorder $\prefeq$,
we need stronger conditions, than those of Theorem \ref{t:affine-quantile}. Two issues are important
in this respect:
\begin{tightitemize}
\item Continuity of $U(\cdot)$ on an appropriate complete topological vector space containing the set~$\Q$ of quantile functions; and
\item Integral representation of a continuous linear functional on this space.
\end{tightitemize}
The first issue cannot be easily resolved in a way similar to the proof of Theor\-em~\ref{t:vNM2}. Even if we assume
continuity of the preorder $\prefeq$ (in some topology), we can prove continuity of $U(\cdot)$ on $\Q$, but
there is no general way to derive from this the continuity of $U(\cdot)$ on some complete topological vector space
containing $\Q$. That is why, we adopt a different approach and derive continuity from monotonicity.

Consider the algebra $\varSigma$ of all sets obtained by finite unions and intersections of intervals of the form
 $(a,b]$ in $(0,1]$, where $0<a<b\le 1$. We define the space $B\big((0,1],\varSigma\big)$ of all bounded
functions on $(0,1]$ that can be obtained as \emph{uniform limits} of sequences of simple functions.
Recall that a \emph{simple function} is a function of the following form:
 \begin{equation}
 \label{simple-a}
 f(p) = \sum_{i=1}^n \alpha_i \1_{A_i}(p),\quad p\in (0,1],
 \end{equation}
 where $\alpha_i\in \R$ for $i=1,\dots,n$, and $A_i$, $i=1,\dots,n$, are disjoint elements of the field $\varSigma$.
In the formula above, $\1_A(\cdot)$ denotes the characteristic function of a set $A$.

The space $B\big((0,1],\varSigma\big)$, equipped with the supremum norm:
 \[
 \|\varPhi\| = \sup_{0<p\le 1}\varPhi(p),
 \]
 is a Banach space. The reader may consult \cite[Ch. III]{Dunford:1958} for  information
 about integration with respect to a finitely additive measure and spaces of bounded functions.

 From now on, we shall consider
only \emph{compactly supported distributions}, and the prospect space $Q_{\textup{b}}$ of all bounded, nondecreasing, and left-continuous
functions on $(0,1]$.\footnote{Bounded nondecreasing functions on $(0,1)$ can be extended to $(0,1]$ by assigning their left limits as their values at~1.}
The set $Q_{\textup{b}}$ is contained in $B\big((0,1],\varSigma\big)$. Indeed, every monotonic function may have only countably many jumps,
their sizes are summable due to the boundedness of the function, and owing to left-continuity it can be represented as a uniform limit
of simple functions.

For two functions $\varPhi$ and $\varPsi$ in $B\big((0,1],\varSigma\big)$, we write $\varPhi \ge \varPsi$, if
$\varPhi(p)\ge \varPsi(p)$ for all $p\in (0,1)$.

\begin{definition}
\label{d:monotonic_dual}
A preorder $\prefeq$ on $\Q_{\textup{b}}$ is \emph{monotonic} with respect to the partial
ord\-er~$\geq$,
 if for all $\varPhi,\varPsi\in \Q_{\textup{b}}$, the implication
$\varPhi \ge \varPsi \;\Longrightarrow\; \varPhi \prefeq \varPsi$
is true.
\end{definition}


\begin{theorem}
\label{t:quantile-cont}
If a total preorder $\prefeq$ on $\Q_{\textup{b}}$ is continuous, monotonic, and satisfies the dual independence axiom,
then a linear continuous functional
on $B\big((0,1],\varSigma\big)$ exists, whose restriction to $\Q_{\textup{b}}$ is a numerical representation of~$\prefeq$.
\end{theorem}
\begin{proof}
Since the continuity axiom implies the Archimedean axiom, Theorem \ref{t:affine-quantile} implies the existence
of a linear functional $U:\lin(\Q_{\textup{b}})\to\R$ whose restriction to $\Q_{\textup{b}}$ is a numerical representation
of~$\prefeq$. The continuity axiom implies the continuity of the functional $U(\cdot)$  on $\Q_{\textup{b}}$.
We shall extend $U(\cdot)$ to a continuous functional on the entire space $B\big((0,1],\varSigma\big)$.

Every simple function can be expressed as
\[
 \varPhi = \sum_{i=1}^n z_i \1_{(p_i,p_{i+1}]} = \sum_{z_i<0} |z_i|\big( \1_{(p_{i+1},1]}  - \1_{(p_{i},1]}\big)
 + \sum_{z_i>0}z_i\big( \1_{(p_i,1]} -  \1_{(p_{i+1},1]}\big).
\]
with $0=p_1<p_2<\dots<p_{n+1}=1$, and thus is an element of $\lin(\Q_{\textup{b}})$.
Consequently, the linear functional $U(\cdot)$ is well-defined on the space of simple functions. Moreover,
rearranging terms, we see that $\varPhi$ is a difference of two simple functions in $\Q_{\textup{b}}$.

 Since the preorder $\prefeq$ is monotonic,  the linear functional $U(\cdot)$ is monotonic on $\Q_{\textup{b}}$.
 We shall prove that it is also monotonic on the set of simple functions in $B\big((0,1],\varSigma\big)$.
Let $\varPhi$ and $\varPsi$ are two simple functions, and let $\varPhi\ge \varPsi$. Then $\varPhi=\varPhi_1-\varPhi_2$,
 $\varPsi=\varPsi_1-\varPsi_2$, where $\varPhi_1,\varPhi_2,\varPsi_1,\varPsi_2\in \Q_{\textup{b}}$, and
 \[
 \varPhi_1+\varPsi_2 \ge \varPhi_2+\varPsi_1.
 \]
 As both sides are elements of $\Q_{\textup{b}}$ and $U(\cdot)$ is nondecreasing in $\Q_{\textup{b}}$ and linear,
 regrouping the terms we obtain
 \[
 U(\varPhi) - U(\varPsi) = U(\varPhi_1-\varPhi_2-\varPsi_1+\varPsi_2)= U(\varPhi_1+\varPsi_2) - U(\varPhi_2+\varPsi_1) \ge 0.
 \]
 This proves the monotonicity of $U(\cdot)$ on the subspace of simple functions.

  For any function $\varGamma\in B\big((0,1],\varSigma\big)$,  we construct two sequences of simple functions: $\{\varPhi_n\}$ and $\{\varPsi_n\}$
 such that $\varPhi_n \le \varGamma \le \varPsi_n$, for $n=1,2,\dots$,
and
 \[
 \varGamma=\lim_{n\to\infty}\varPhi_n = \lim_{n\to\infty}\varPsi_n.
 \]
 The sequence $\{U(\varPhi_n)\}$ is bounded from above by $U(\varPsi_k)$ for any $k$, due to the monotonicity of $U(\cdot)$.
 Similarly, the sequence $\{U(\varPsi_n)\}$ is bounded from below by $U(\varPhi_k)$ for any $k$.
 Moreover,
  \begin{align*}
 0 &\le U(\varPsi_n) - U(\varPhi_n) = U(\varPsi_n - \varPhi_n) \\
 &\le U\big(\|\varPsi_n - \varPhi_n\|\1_{(0,1]}\big)
 = U(\1_{(0,1]})\|\varPsi_n - \varPhi_n\| \to 0.
 \end{align*}
 Therefore, both sequences $\{U(\varPhi_n)\}$ and $\{U(\varPsi_n)\}$ have the same limit and we can define
 \[
 U( \varGamma)=\lim_{n\to\infty}U(\varPhi_n) = \lim_{n\to\infty}U(\varPsi_n).
 \]

We may use any sequence of simple functions $\varGamma_n\to \varGamma$ to calculate $U(\varGamma)$. Indeed,
 setting $\varPhi_n = \varGamma_n - \|\varGamma_n-\varGamma\|$ and $\varPsi_n = \varGamma_n + \|\varGamma_n-\varGamma\|$, we obtain
 $\varPhi_n \le \varGamma \le \varPsi_n$ and
 $\varPhi_n \le \varGamma_n \le \varPsi_n$. Consequently,  $U(\varPhi_n) \le U(\varGamma_n) \le U(\varPsi_n)$ and
 \[
 \lim_{n\to\infty}U(\varGamma_n) = U(\varGamma).
 \]
The functional $U:B\big((0,1],\varSigma\big)\to\R$ defined in this way is linear on the subspace of simple functions, which is a subspace of  $\lin(\Q_{\textup{b}})$.
Consider two elements $\varPhi$ and $\varPsi$ of $B\big((0,1],\varSigma\big)$, and two sequences $\{\varPhi_n\}$ and $\{\varPsi_n\}$ of simple functions such that
$\varPhi_n\to\varPhi$ and $\varPsi_n\to\varPsi$. For any $a,b\in\R$, we obtain
\begin{align*}
U(a\varPhi+b\varPsi) &= \lim_{n\to\infty}U(a\varPhi_n+b\varPsi_n) = \lim_{n\to\infty}\big[ aU(\varPhi_n)+b U(\varPsi_n)\big] \\
&=  a \lim_{n\to\infty} U(\varPhi_n)+b \lim_{n\to\infty} U(\varPsi_n) = aU(\varPhi)+b U(\varPsi).
\end{align*}
This proves the linearity of $U(\cdot)$ on the whole space $B\big((0,1],\varSigma\big)$.

To verify monotonicity, consider two elements $\varPhi\le \varPsi$ in $B\big((0,1],\varSigma\big)$, and two sequences $\{\varPhi_n\}$ and $\{\varPsi_n\}$ of simple functions such that
$\varPhi_n\to\varPhi$, $\varPsi_n\to\varPsi$, and $\varPhi_n \le \varPhi \le \varPsi \le \varPsi_n$. As $U(\cdot)$ is monotonic
on the space of simple functions, we obtain $U(\varPhi_n) \le U(\varPsi_n)$, and thus $U(\varPhi) \le U(\varPsi)$.

To prove continuity, consider any element $\varPhi\in B\big((0,1],\varSigma\big)$. Owing to linearity and monotonicity of $U(\cdot)$, we obtain
 \begin{align*}
 U(\varPhi)
 \le U\big(\|\varPhi\|\1_{(0,1]}\big)
 = U(\1_{(0,1]})\|\varPhi\|.
 \end{align*}
Consequently, $U(\cdot)$ is continuous.
\end{proof}

Now, we can prove the main result of this section. It involves integration with respect to finitely additive measures, which we denote by the symbol $\ds$. To the best of our knowledge,
it is original in its formulation and derivation.
\begin{theorem}
\label{t:quantile-int}
Suppose the total preorder $\prefeq$ on $\Q_{\textup{b}}$ is continuous, monotonic, and satisfies the independence axiom.
Then a nonnegative, bounded, finitely additive measure $\mu$ on $\varSigma$ exists, such that the functional
\begin{equation}
\label{quantile-intmu}
U(\varPhi) = \int_0^1 \varPhi(p)\;\ds\mu,\quad \varPhi\in B\big((0,1],\varSigma\big),
\end{equation}
 is a numerical representation of $\,\prefeq$.
\end{theorem}
\begin{proof}
The functional $U:B\big((0,1],\varSigma\big)\to\R$ constructed in the proof of Theorem~\ref{t:quantile-cont} is linear and continuous.
By virtue of \cite[Theorem IV.5.1]{Dunford:1958}, it has the form (\ref{quantile-intmu}) of an integral with respect to
a certain bounded and finitely additive meas\-ure~$\mu$. As $U(\cdot)$ is nondecreasing, $\mu$ is nonnegative.
\end{proof}

Under additional conditions, we can write the integral \eqref{quantile-intmu} in a more familiar form of a Stieltjes integral.
We define a nondecreasing and bounded function $w:[0,1]\to \R_{+}$ as follows:
\begin{equation}
\label{wdist}
w(p) = \mu\big((0,p]\big),\quad p\in(0,1];\quad w(0)=0.
\end{equation}
If the jump points of $\varPhi(\cdot)$ and $w(\cdot)$ do not coincide, we can rewrite \eqref{quantile-intmu} as follows:
\begin{equation}
\label{quantile-int}
U(\varPhi) = \int_0^1 \varPhi(p)\; d w(p).
\end{equation}
In general, however, to validate the integral representation \eqref{quantile-int}, we need a weaker topology
on the prospect space.
We use the $\Lc_1$-topology on the space $\Q_{\textup{b}}$ of quantile functions,
defined by the distance function
\[
\dist(\varPhi,\varPsi) = \int_0^1 |\varPhi(p)-\varPsi(p)|\;dp.
\]
\begin{theorem}
\label{t:rank-dependent}
Suppose the total preorder $\prefeq$ on $\Q_{\textup{b}}$ is monotonic, continuous in the $\Lc_1$-topology,
and satisfies the independence axiom.
Then a bounded, nondecreasing, and continuous function $w:[0,1]\to \R$ exists, such that the functional
\eqref{quantile-int}
 is a numerical representation of $\,\prefeq$.
\end{theorem}
\begin{proof}
The assumptions of Theorem \ref{t:quantile-int} are satisfied, and thus a finitely additive measure $\mu$ exists such that formula
\eqref{quantile-intmu} holds.
Define $w(\cdot)$ by \eqref{wdist}. As $\mu$ is nonnegative, $w(\cdot)$ is nondecreasing.

Consider a sequence of simple functions
$\1_{(p_n,1]}$, with $p_n\to p \in (0,1)$, as $n\to\infty$. They are elements of $\Q_{\textup{b}}$ and converge
in the $\Lc_1$-topology  to $\1_{(p,1]}$. The continuity of the preorder
$\prefeq$ in this topology implies that the numerical representation \eqref{quantile-intmu} is continuous. We obtain
\[
U\big(\1_{(p_n,1]}\big) = \mu\big((p_n,1]\big) = w(1) - w(p_n) \to U\big(\1_{(p,1]}\big) = w(1)-w(p).
\]
Thus $w(\cdot)$ is continuous in $(0,1)$.
If $p=1$, then $\1_{(p_n,1]}\to 0$,  and we obtain in the same way $w(p_n)\to w(1)$.

As $w(\cdot)$ is  continuous, $\mu$ is a regular, bounded,
countably additive, and atomless measure. Consequently,
the integral representation \eqref{quantile-intmu} can be written as a Stieltjes integral \eqref{quantile-int}.
\end{proof}

The function $w(\cdot)$ appearing in the integral representation \eqref{quantile-int} is called the \emph{rank-depend\-ent utility function} or
\emph{dual utility function}.
\begin{remark}
\label{r:guriev}
{\rm
An attempt to derive an even stronger representation, with a density of $w(\cdot)$ with respect to the Lebesgue measure, has been made in
\cite[Thm. 1]{guriev}. Unfortunately, the proof of that theorem contains an incorrigible error (lines 14--15 on page 132).
}
\end{remark}
\begin{remark}
\label{r:Quiggin}
{\rm
In a fundamental contribution, Quiggin \cite{Quiggin:1982} considers discrete distributions and derives from a different system of axioms
the existence of an \emph{anticipated utility functional}. In our notation, for a simple quantile function
$ \varPhi = \sum_{i=1}^n z_i \1_{(p_{i-1},p_{i}]}$
with $z_1\le z_2 \le \dots \le z_n$, and with cumulative probabilities $0=p_0 \le p_1 \le \dots \le p_n=1$,
 this functional has the form
 \begin{equation}
 \label{AU}
 U(\varPhi) = \sum_{i=1}^n u(z_i)\big[ w(p_i) - w(p_{i-1})\big],
 \end{equation}
 where $u(\cdot)$ and $w(\cdot)$ are nondecreasing functions (see \cite[Ch. 11]{Quiggin:1993} and \cite[Thm. 3.2]{QuigginWakker} and the references therein). This corresponds to \eqref{quantile-int} with $u(z)\equiv z$. The term \emph{rank-dependent utility function},
 which we adopt for $w(\cdot)$, is borrowed from this theory.
}
\end{remark}

\subsection{Choquet Integral Representation of Dual Utility}
\label{s:choquet}

We presented the dual utility theory in the prospect space of quantile functions, which is most natural for it. It is interesting,
though, to derive an equivalent representation in the prospect space of distribution functions. Every bounded quantile
function $\varPhi\in \Q_{\textup{b}}$ corresponds to a distribution function $F:\R \to [0,1]$ of a measure with bounded support,
\begin{align*}
\varPhi(p)  &= F^{-1}(p) \eqdef \inf\,\{ \eta\in \R: F(\eta) \ge p\},\\
F(z) &=  \varPhi^{-1}(z) \eqdef \begin{cases}
 \sup\,\{ p\in [0,1]: \varPhi(p) \le z\} & \text{if} \ z \ge \varPhi(0),\\
0 & \text{otherwise.}
\end{cases}
\end{align*}
The following theorem employs a form of integration by parts for the integral \eqref{quantile-intmu} and corresponds
to the representation derived in \cite{Schmeidler}.
\begin{theorem}
\label{t:schmeidler}
Suppose the total preorder $\prefeq$ on $\Q_{\textup{b}}$ is monotonic, continuous with respect to uniform convergence
and satisfies the dual independence axiom.
Then  a nondecreasing function $w:[0,1]\to[0,1]$ exists, satisfying $w(0)=0$ and $w(1)=1$,
and such that the functional
\begin{equation}
\label{schmeidler}
 U(F^{-1})= -\int_{-\infty}^0 w\big(F(z)\big) \;dz + \int_0^{\infty}\big[ 1 - w\big(F(z)\big)\big]\;dz
\end{equation}
 is a numerical representation of $\,\prefeq$.
\end{theorem}
\begin{proof}
Due to Theorem \ref{t:quantile-int}, the functional \eqref{quantile-intmu}
with some nonnegative, bounded, finitely additive measure $\mu$ on $\varSigma$, is a numerical representation of $\prefeq$.
Without loss of generality, we may assume that $\mu\big( (0,1]\big)=1$ and define the function $w(\cdot)$ by~\eqref{wdist}.

First, we check the formula \eqref{schmeidler} for a stepwise function $\varPhi(\cdot)$, given as follows:
 \begin{equation}
 \label{simple-B}
 \varPhi(x) = \sum_{i=1}^n z_i \1_{(p_i,p_{i+1}]}(x),\quad x\in (0,1],
 \end{equation}
 where $z_1\le z_2 \le \dots \le z_{n}$ and $0=p_1\le p_2 \le \dots \le p_n \le p_{n+1}=1$.
In the formula above, $\1_A(\cdot)$ denotes the characteristic function of a set $A$.

The integral \eqref{quantile-intmu} takes on the form:
\[
\int_0^1 \varPhi(p)\;\ds\mu = \sum_{i=1}^n z_i\mu\big((p_{i},p_{i+1}]\big)= \sum_{i=1}^n z_i\big[ w(p_{i+1})-w(p_i)\big].
\]
With no loss of generality we may assume that $z_k=0$ for some $k$. Then we can continue the above relations as follows:
\begin{align*}
\int_0^1 \varPhi(p)\;\ds\mu  &= \sum_{i=1}^{k-1} w(p_{i+1})(z_i-z_{i+1}) + \sum_{i=k}^{n-1}\big[ 1 - w(p_{i+1})\big](z_{i+1}-z_{i})\\
&= -\sum_{i=1}^{k-1} w\big(F(z_i)\big)(z_{i+1}-z_{i}) + \sum_{i=k}^{n-1}\big[ 1 - w\big(F(z_i)\big)\big](z_{i+1}-z_{i}).
\end{align*}
This proves the formula \eqref{schmeidler} for simple functions.

To prove it for a general function $\varPhi\in \Q_{\textup{b}}$, we consider two sequences of simple functions $\{\varPhi_n\}$ and $\{\varPsi_n\}$,
such that $\varPhi_n \le \varPhi \le \varPsi_n$, $n=1,2,\dots,$ and
\[
\|\varPsi_n-\varPhi_n\|_{\infty}\to 0\quad \text{as}\quad n\to\infty.
\]
Since $\varPhi_n^{-1}\ge \varPhi^{-1} \ge \varPsi_n^{-1}$ and $w(\cdot)$ is nondecreasing, we have
\begin{align*}
 U(\varPhi_n) &= -\int_{-\infty}^0 w\big(\varPhi_n^{-1}\big) \;dz + \int_0^{\infty}\big[ 1 - w\big(\varPhi_n^{-1}\big)\big]\;dz\\
 &\le  -\int_{-\infty}^0 w\big(\varPhi^{-1}\big) \;dz + \int_0^{\infty}\big[ 1 - w\big(\varPhi^{-1}\big)\big]\;dz \\
 &\le  -\int_{-\infty}^0 w\big(\varPsi_n^{-1}\big) \;dz + \int_0^{\infty}\big[ 1 - w\big(\varPsi_n^{-1}\big)\big]\;dz
  = U(\varPsi_n).
\end{align*}
The first and the last equation follow from the formula \eqref{schmeidler} for simple functions. Since $U(\cdot)$ is continuous,
the leftmost and the rightmost members converge to $U(\varPhi)$, as $n\to\infty$, and thus the middle member must be equal to $U(\varPhi)$.
\end{proof}
\begin{remark}
\label{r:schmeidler}
{\rm
The assertion of Theorem \ref{t:schmeidler}, in the case of distributions supported on $[0,1]$, is similar to
the assertion of \cite[Thm. 1]{Yaari:1987}. Our assumptions are weaker, however. We do not assume any uniform bound on all
quantile functions in the prospect space, and we assume continuity of the preorder $\prefeq$
with respect to the topology of uniform convergence,
rather than with respect to $\Lc_1$-topology, required in \cite[A3]{Yaari:1987}. Therefore, we could not resort to the expected utility
theory applied to the quantile functions, as in the proof of \cite[Thm.~1]{Yaari:1987}.
}
\end{remark}
\begin{remark}
\label{r:choquet}
{\rm
Formula \eqref{schmeidler} is a special case of the Choquet integral of the function $F^{-1}(\cdot)$ (see  \cite{Choquet}).
In our case we did not invoke the theory of capacities, because the prospect space contains only monotonic functions.
}
\end{remark}
\subsection{Risk Aversion}
\label{s:ra-quantile}

For every $\varPhi\in \Q_{\textup{b}}$ and any $\sigma$-subalgebra $\Gc$ of the Borel $\sigma$-algebra $\B$ on $\R$, the conditional expectation $\Ec_{\varPhi|\Gc}$ is defined as a $\Gc$-measurable function, satisfying the equation
\[
\int\limits_G \Ec_{\varPhi|\Gc}(z)\;d\varPhi^{-1}(z) = \int\limits_G z\;d\varPhi^{-1}(z),\quad  G\in \Gc.
 \]
 Observe that it is sufficient to require this equation for the smallest collection $\Jc$
 of intervals of form $(-\infty,c]$, generating $\Gc$:
 \[
\int_{-\infty}^c \Ec_{\varPhi|\Gc}(z)\;d\varPhi^{-1}(z)
= \int_{-\infty}^c z\;d\varPhi^{-1}(z),\quad  \forall\; (\infty,c]\in \Jc.
 \]
The corresponding quantile function of
$\Ec_{\varPhi|\Gc}$, denoted by $\varPhi_{\Gc}(p)$, is $\varPhi^{-1}(\Gc)$-measurable and satisfies the equation
\[
\int\limits_{\varPhi^{-1}\big((\infty,c]\big)}\hspace{-1em} \varPhi_{\Gc}(p)\;dp =
\int\limits_{\varPhi^{-1}\big((\infty,c]\big)} \hspace{-1em}\varPhi(p)\;dp,\quad \forall\; (\infty,c]\in \Jc.
\]
This equation can be rewritten as follows:
\[
\int_0^{\beta} \varPhi_{\Gc}(p)\;dp = \int_0^{\beta} \varPhi(p)\;dp,\quad \forall\;\beta\in \varPhi\big((0,1]\big).
\]

%

\begin{definition}
\label{d:risk-averse-quantile}
A preference relation $\prefeq$ on $\Q_{\textup{b}}$ is \emph{risk-averse}, if
$\varPhi_{\Gc} \prefeq \varPhi$,
for every $\varPhi\in \Q_{\textup{b}}$ and every $\sigma$-subalgebra $\Gc$ of the Borel $\sigma$-algebra $\B$ on $\R$.
\end{definition}

\begin{theorem}
\label{t:dual-ra}
Suppose a total preorder $\prefeq$ on $\Q_{\textup{b}}$ is monotonic, continuous,  and satisfies the dual independence axiom.
Then it is risk averse if and only if it has the  numerical representation  \eqref{schmeidler} with
a nondecreasing and concave function $w:[0,1]\to[0,1]$.
\end{theorem}
\begin{proof} In view of Theorem \ref{t:schmeidler}, we only need to prove that $w(\cdot)$ is concave.
Consider any $0<p_1<p_2 < p_3\le 1$.
Define a four-point distribution with mass $p_1$ at $-3$, mass $p_2-p_1$ at $-2$, mass $p_3-p_2$ at $-1$,
and mass $1-p_3$ at $0$.
The corresponding quantile function $\varPhi$ has, according to \eqref{schmeidler}, the utility
$U(\varPhi) = -\big[w(p_1)+w(p_2)+w(p_3)\big]$. For a $\sigma$-subalgebra generated by $G_1=(-\infty,-3]$
 and $G_2=(-\infty,-1]$, the corresponding
conditional expectation has value $-3$ with probability $p_1$, value $(2p_1-p_2-p_3)/(p_3-p_1)$
with probability $p_3-p_1$, and value 0 with probability $1-p_3$.
The corresponding quantile function $\varPhi_{\Gc}$ has the utility
\[
U(\varPhi_{\Gc}) = -\bigg[ w(p_1)\frac{-p_1-p_2+2p_3}{p_3-p_1} + w(p_3)\frac{-2p_1+p_2+p_3}{p_3-p_1}\bigg].
\]
Owing to  risk aversion, $U(\varPhi_{\Gc}) \ge U(\varPhi)$. After elementary manipulations, we obtain the inequality
\[
w(p_1)\frac{p_3-p_2}{p_3-p_1} + w(p_3)\frac{p_2-p_1}{p_3-p_1} \le w(p_2).
\]
Let $\alpha\in (0,1)$ and let $p_2=\alpha p_1 + (1-\alpha)p_3$. Then the last inequality reads:
\[
\alpha w(p_1) + (1-\alpha) w(p_2) \le w\big( \alpha p_1 + (1-\alpha)p_3\big).
\]
This is equivalent to the concavity of $w(\cdot)$ on $(0,1]$.
\end{proof}
\begin{remark}
\label{r:E-insuf}
{\rm
If we assumed only that the quantile function of the expected value is preferred, that is
$\varPhi_{\Gc_0} \prefeq \varPhi$, where $\Gc_0=\{\R,\emptyset\}$ is the trivial $\sigma$-subalgebra, then we would not be able to infer the concavity of the
function $w(\cdot)$.

In fact, the relation $\varPhi_{\Gc_0} \prefeq \varPhi$ for all
$\varPhi$ is equivalent to
the inequality $w(p)\ge p$ for all $p\in (0,1]$. Indeed, consider a two-point distribution, with mass $p$ at 0 and mass $1-p$ at 1. Then
$U(\varPhi) = 1-w(p)$ and $U\big(\varPhi_{\Gc_0}\big) = 1-p$. Thus $w(p)\ge p$.
To prove the converse implication, we use the inequality $w\big(F(z)\big)\ge F(z)$ in \eqref{schmeidler} to obtain
\[
 U(\varPhi)\le -\int_{-\infty}^0 F(z) \;dz + \int_0^{\infty}\big[ 1 - F(z)\big]\;dz = U\big(\varPhi_{\Gc_0}\big),
\]
as required.

This is in contrast to the expected utility case, when preference of the expected value was sufficient to
derive preference of all conditional expectations (\emph{cf.} Remark \ref{r:eisk-aversion}).
}
\end{remark}

\section{Preferences Among Random Vectors}
\label{s:vectors}

\subsection{Expected Utility Theory for Random Vectors}
\label{s:utility-vectors}

Suppose $(\varOmega,\F,P)$ is a probability space and the prospect space $\Z$ is the space of random vectors $Z:\varOmega\to \Sc$,
where $\Sc$ is a Polish space equipped with its Borel $\sigma$-algebra $\B$.
The distribution of a random vector $Z\in \Z$ is the probability measure $\mu_Z$ on $\B$ defined as $\mu_Z=P\comp Z^{-1}$.
We say that $Z$ and $D$ \emph{have the same law} and write $Z\Dsim W$, if $\mu_Z=\mu_W$.

The preference relation $\prefeq$ on $\Z$ is called \emph{law invariant} if
$Z \Dsim W$ implies that $Z \sim W$.
Every preference relation $\succeq$ on $\Pc(\Sc)$, the space of probability measures on $\Sc$, defines a
law invariant preference relation $\prefeq$ on $\Z$ as follows:
\[
Z \prefeq W \iff \mu_Z \succeq \mu_W.
\]
The converse statement is true, if we additionally require that every probability measure $\mu$ on $\Sc$ is a distribution of some $Z\in \Z$.
This can be guaranteed if $(\varOmega,\F,P)$ is a standard atomless probability space (see \cite[Thm. 11.7.5]{Dudley} and \cite{Skorohod}).
In this case, we can consider an operation on random variables
in $\Z$ corresponding to the operation of taking a convex combination of measures on $\Sc$.

For three elements $Z$, $V$, and $W$ in $\Z$ we say that  $W$ is a \emph{lottery} of $Z$ and $V$ with probabilities $\alpha \in (0,1)$ and $(1-\alpha)$,
if an event $A\in \F$ of probability $\alpha$ exists,
such that the conditional distribution of $W$, given $A$, is the same as the (unconditional) distribution of $Z$;
while the conditional distribution of $W$, given $\widebar{A}=\varOmega\setminus A$, is the same as the unconditional distribution of $V$.
In this case, the probability measure $\mu_W$ induced by $W$ on $\Sc$ is the corresponding convex combination of the probability measures $\mu_Z$ and $\mu_V$ of $Z$ and $V$, respectively:
\[
\mu_W = \alpha \mu_Z + (1-\alpha) \mu_V.
\]
We write the lottery symbolically as
\[
W = \alpha Z \lott (1-\alpha) V.
\]
It should be stressed that only the distribution $\mu_W$ of the lottery is defined uniquely, not the random variable $W$ itself. However, if the preference
relation $\prefeq$ on $\Z$ is law invariant, it makes sense to compare lotteries.

For law invariant preferences on the space of random vectors with values in $\Sc$, we  introduce axioms corresponding to the axioms
of the expected utility theory for distributions.
\begin{description}
\item[\emph{Independence Axiom for Random Vectors:}] For all $Z,V,W \in \Z$ one has
\[
Z \pref V \;\Longrightarrow\;
\alpha Z \lott (1-\alpha) W \pref \alpha V \lott (1-\alpha) W,\quad \forall\, \alpha\in (0,1),
\]
\item[\emph{Archimedean Axiom for Random Vectors:}] If $Z \pref V \pref W$, then  $\alpha,\beta\in (0,1)$ exist such that
\[
\alpha Z \lott (1-\alpha) W \pref V \pref \beta Z \lott (1-\beta) W.
\]
\end{description}
These conditions allow us to reproduce the results of Section \ref{s:affine-utility} in the language of random vectors.
 Directly from Theorem \ref{t:vNM1} we obtain the following result.
\begin{corollary}
\label{t:affine-utility-vectors}
Suppose the total preorder $\prefeq$ on $\Z$ satisfies the independence and Archimedean axioms for random vectors. Then  a numerical
representation $U:\Z\to\R$ of $\,\prefeq$ exists, which satisfies for all $Z,V\in \Z$ and all $\alpha\in [0,1]$ the equation
\[
U(\alpha Z \lott (1-\alpha) V) = \alpha U(Z)+(1-\alpha) U(V).
\]
\end{corollary}

In order to invoke the integral representation from \S \ref{s:integral-utility}, we need to introduce an appropriate topology on the space $\Z$
and assume continuity of the preorder $\prefeq$ in this topology. For this purpose we adopt the topology of convergence in distribution.
Recall that a sequence of random vectors $Z_n:\varOmega\to\Sc$ converges in distribution to a random vector $Z:\varOmega\to\Sc$, if
the sequence of probability measures $\{\mu_{Z_n}\}$ converges weakly to the measure $\mu_Z$.

We  can now recall Theorem \ref{t:vNM2} to obtain an integral representation of the utility functional.
\begin{corollary}
\label{t:integral-utility-vectors}
Suppose the total preorder $\prefeq$ on $\Z$ is law invariant, continuous, and satisfies the independence axiom for random vectors.
Then a continuous and bounded function $u:\Sc\to \R$ exists, such that the functional
\begin{equation}
\label{vNMrep-vec}
U(Z) = \Eb\big[u(Z)\big] = \int_{\varOmega}u\big(Z(\omega)\big)\;P(d\omega)
\end{equation}
is a numerical representation of $\,\prefeq$ on $\Z$.
\end{corollary}
It should be stressed, however, that the assumption of continuity with respect to the topology of weak convergence is rather strong.
For example, if we assume only that for every $Z \in \Z$ the sets
\[
\{V\in \Z: V \prefeq Z\}\quad\text{and} \quad \{V\in \Z: Z\prefeq V\}
\]
are closed in the space $\Lc_1(\varOmega,\F,P;\Sc)$,
the existence of a utility function is not guaranteed.

Monotonicity and risk aversion considerations from section \ref{s:ra-mu} translate to the case of random vectors in a straightforward way.

Suppose $\Sc$ is a separable Banach space, with a partial order relation $\ge$.
In the definition below, the relation $\ge$ applied to random vectors is understood in the almost sure sense.

\begin{definition}
\label{d:monotonic-vector}
The total preorder $\prefeq$ is called monotonic with respect to the partial order $\ge$, if $Z \ge V \;\Longrightarrow\; Z  \prefeq V$.
\end{definition}

In this section, we shall always understand the monotonicity of a preorder $\prefeq$ in the sense
of Definition \ref{d:monotonic-vector}.


The following result is a direct consequence of Theorem \ref{t:vNM-mono}.
\begin{corollary}
\label{c:vNM-mono}
Suppose the total preorder $\prefeq$ on $\Z$ is monotonic, continuous, and satisfies the independence axiom for random vectors. Then a nondecreasing, continuous and
bounded function $u:\Sc\to \R$ exists, such that the functional \eqref{vNMrep-vec}
is a numerical representation of $\,\prefeq$ on $\Z$.
\end{corollary}

We now focus on the case when 
every $Z\in\Z$ the Bochner integral (the expected value)
\[
\Eb[Z]= \int_\varOmega Z(\omega)\;P(d\omega),
\]
is well-defined (for integration of Banach space valued
 random vectors, see \cite[{\S} II.2]{DiestelUhl}).
 Then for every $\sigma$-subalgebra $\Gc$ of $\F$ the conditional expectation
$\Eb[Z|\Gc]:\varOmega\to\Sc$ is defined as a $\Gc$-measurable function satisfying
\[
\int_G \Eb[Z|\Gc](\omega)\;P(d\omega) = \int_G Z(\omega)\;P(d\omega), \quad \forall\;G\in\Gc,
 \]
(see,\emph{ e.g.}, \cite[{\S}2.1]{LedouxTalagrand}).

\begin{definition}
\label{d:risk-averse-vec}
A preference relation $\prefeq$ on $\Z$ is \emph{risk-averse}, if
$\Eb[Z|\Gc] \prefeq Z$,
for every $Z\in \Z$ and every $\sigma$-subalgebra $\Gc$ of $\F$.
\end{definition}

The following corollary is a direct consequence of Remark \ref{r:eisk-aversion}, because
Definition \ref{d:risk-averse-vec} implies that $\Eb[Z] \prefeq Z$.
\begin{corollary}
\label{c:vNM2psi-ra}
Suppose a total preorder $\prefeq$ on $\Z$ is continuous, risk-averse, and satisfies the independence axiom for random vectors. Then a concave function
$u :\Sc \to \R$ exists such that the functional \eqref{vNMrep-vec}
is a numerical representation of $\,\prefeq$ on $\Z$.
\end{corollary}

Again, as discussed in Remark \ref{r:eisk-aversion}, it would be sufficient to assume that $\Eb[Z]\prefeq Z$
for all $Z\in \Z$, but we shall need Definition \ref{d:risk-averse-vec} also in the next subsection, where such
simplification will not be justified.

\subsection{Dual Utility Theory for Random Variables}
\label{s:dual-utility-variables}

The dual utility theory can be formulated in the prospect space $\Z$ of real-valued random variables defined on a probability space $(\varOmega,\F,P)$.
The axioms formulated in section \ref{s:dual-prospects} for quantile functions can be equivalently formulated for comonotonic
random variables.
Recall that real random variables $Z_i$, $i=1,\dots,n$, are \emph{comonotonic}, if
\[
\big(Z_i(\omega)-Z_i(\omega')\big)\big(Z_j(\omega)-Z_j(\omega')\big) \ge 0
\]
for all $\omega,\omega'\in\varOmega$ and all $i,j=1,\dots,n$.

The following axioms were formulated in \cite{Yaari:1987}, when the theory of dual utility was axiomatized.
\begin{description}
\item[\emph{Dual Independence Axiom for Random Variables:}] For all comonotonic random variables $Z$, $V$, and $W$  in $\Z$ one has
\[
Z \pref V \;\Longrightarrow\;
\alpha Z + (1-\alpha) W \pref \alpha V + (1-\alpha) W,\quad \forall\, \alpha\in (0,1),
\]
\item[\emph{Dual Archimedean Axiom for Random Variables:}] For all comonotonic random variables $Z$, $V$, and $W$  in $\Z$, satisfying the relations
\[
Z \pref V \pref W,
\]
there exist $\alpha,\beta\in (0,1)$ such that
\[
\alpha Z + (1-\alpha) W \pref V \pref \beta Z + (1-\beta) W.
\]
\end{description}
In addition to that, in \cite{Yaari:1987} the preorder $\prefeq$ was assumed monotonic
in the sense of Defi\-ni\-tion~\ref{d:monotonic-vector}.

It is clear that for comonotonic random variables the first two axioms are equivalent to the axioms discussed in \S \ref{s:dual-prospects}. Furthermore, if a preorder $\prefeq$ is monotonic in the sense of
Defi\-ni\-tion~\ref{d:monotonic-vector}, then the corresponding preorder on the space of quantile
 functions is monotonic in the sense of Definition \ref{d:monotonic_dual}.

\begin{theorem}
\label{t:affine-dual-variables}
Suppose $\Z$ is the set of random variables on a standard and atomless probability space $(\varOmega,\F,P)$.
If the total preorder $\prefeq$ on $\Z$ is law invariant, and satisfies the dual independence and Archimedean axioms for random variables,
then a numerical
representation $U:\Z\to\R$ of $\,\prefeq$ exists, which satisfies for all comonotonic $Z,V\in \Z$ and all $\alpha,\beta\in \R_{+}$ the equation
\begin{equation}
\label{comonotonic-additivity}
U( \alpha Z + \beta V) =  \alpha U(Z)+ \beta U(V).
\end{equation}
Moreover,
\begin{equation}
\label{dual-certainty}
U(c\1) = c,\quad \forall\,c\in \R.
\end{equation}
\end{theorem}
\begin{proof}
Let $Y$ be a uniform random variable on $(\varOmega,\F,P)$.
The preference relation $\prefeq$ on $\Z$ induces a preference relation $\succeq$ on $\Q$ by the formula
\[
\varPhi \succeq \varPsi \iff \varPhi(Y) \prefeq \varPsi(Y).
\]
The preference relation $\succeq$ does not depend on the particular choice of $Y$, because $\prefeq$ is law invariant.

For comonotonic random variables $Z$ and $V$, and for $\alpha\in (0,1)$, we have
\[
F_{\alpha Z+(1-\alpha)V}^{-1} = \alpha F_Z^{-1} + (1-\alpha) F_V^{-1}.
\]
Thus, the dual independence and Archimedean axioms for the relation $\prefeq$ among random variables imply the same properties for the
relation $\succeq$ on $\Q$. By virtue of Theorem \ref{t:affine-quantile}, a linear functional $\mathcal{U}:\lin(\Q)\to\R$ exists,
whose restriction to $\Q$ is a numerical representation of the preorder~$\succeq$.

Then $U(Z)=\mathcal{U}\big(F_Z^{-1}\big)$, $Z\in\Z$, is a numerical representation of $\prefeq$.
For comonotonic $Z,V\in\Z$ and  $\alpha,\beta \ge 0$, the linearity of $\mathcal{U}$ yields
\begin{align*}
U(\alpha Z +  \beta V) &= \mathcal{U}\big(F_{\alpha Z+\beta V}^{-1}\big)
 =  \mathcal{U}\big(\alpha F_{Z}^{-1} + \beta F_{V}^{-1}\big)\\
 & = \alpha \mathcal{U}\big(F_{Z}^{-1}\big) + \beta \mathcal{U}\big(F_{V}^{-1}\big)
= \alpha U(Z)+ \beta U(V),
\end{align*}
which proves \eqref{comonotonic-additivity}.

By monotonicity, $U(\1)=\mathcal{U}(F_{\1}^{-1}) >0 $. We may normalize $U(\cdot)$ to have $U(\1)=1$. For $c>0$
the equation \eqref{dual-certainty}  follows from \eqref{comonotonic-additivity}. Then
\[
U(-c\1) = \mathcal{U}(F_{-c\1}^{-1}) =
\mathcal{U}(-F_{c\1}^{-1}) = -\mathcal{U}(F_{c\1}^{-1}) = -U(c\1) = -c,
\]
owing to the linearity of $\mathcal{U}(\cdot)$.
\end{proof}

In our further considerations, we assume that $\Z$ is the space of bounded random variables equipped with the
the norm topology of the space $\Lc_1(\varOmega,\F,P)$.
\begin{theorem}
\label{t:rank-dependent-variables-L1}
Suppose $\Z$ is the set of bounded random variables on a standard and atomless probability space $(\varOmega,\F,P)$.
If the total preorder $\prefeq$ on $\Z$ is law invariant, continuous, monotonic,
and satisfies the dual independence axioms for random variables,
then a bounded, nondecreasing, and continuous function $w:[0,1]\to \R_{+}$ exists, such that the functional
\begin{equation}
\label{quantile-int-var}
U(Z) = \int_0^1 F_Z^{-1}(p)\; d w(p),\quad Z\in \Z,
\end{equation}
 is a numerical representation of $\,\prefeq$.
\end{theorem}
\begin{proof}
Recall that the preorder $\prefeq$ induces a preorder $\succeq$ on the space $\Q_{\textup{b}}$ of bounded quantile functions. The preorder $\succeq$
is defined in the proof of Theorem \ref{t:affine-dual-variables}.
It satisfies the monotonicity condition on $\Q_{\textup{b}}$, because for a uniform random variable $Y$ we have the chain of
equivalence relations:
\begin{equation}
\label{equiv-orders}
\varPhi \ge \varPsi \iff \varPhi(Y) \ge \varPsi(Y) \Longrightarrow \varPhi(Y) \prefeq \varPsi(Y) \iff \varPhi \succeq \varPsi.
\end{equation}
The dual independence axiom for $\succeq$ follows from the dual independence axiom for~$\prefeq$ with comonotonic random variables.
In order to use Theorem \ref{t:rank-dependent}, we only need to verify the continuity condition for $\succeq$.

Consider a convergent sequence of functions $\{\varPhi_n\}$ and a function $\varPsi$ in $\Q_{\textup{b}}$, such that
$\varPhi_n \succeq \varPsi$, $n=1,2\dots$, and let $\varPhi$ be the $\Lc_1$-limit of  $\{\varPhi_n\}$, that is,
\[
\lim_{n\to\infty} \int_0^1 |\varPhi_n(p)-\varPhi(p)|\;dp =0.
\]
For a uniform random variable $Y$, we define $Z_n=\varPhi_n(Y)$, $Z=\varPhi(Y)$, and $V=\varPsi(Y)$. By \eqref{equiv-orders},
$Z_n\prefeq V$. Substituting the definitions of $Z_n$ and $Z$ and changing variables we obtain
\begin{align*}
\|Z_n-Z\|_1 &= \int_{\varOmega}|Z_n(\omega)-Z(\omega)|\;P(d\omega) = \int_{\varOmega}|\varPhi_n(Y(\omega))-\varPhi(Y(\omega))|\;P(d\omega)  \\
&=\int_0^1 |\varPhi_n(p)-\varPhi(p)|\;dp \to 0, \quad \text{as}\quad n\to \infty.
\end{align*}
By the continuity of $\prefeq$ in $\Z$, we conclude that $Z\prefeq V$. By \eqref{equiv-orders}, $\varPhi\succeq\varPsi$. In a similar way we consider
the case when $\varPsi \succeq \varPhi_n$, $n=1,2\dots$ and we prove that $\varPsi \succeq \varPhi$. Consequently, the preorder $\succeq$ is continuous in $\Q_{\textup{b}}$.

By Corollary \ref{t:rank-dependent}, a numerical representation $\mathcal{U}(\cdot):\Q_{\textup{b}}\to\R$ of $\succeq$ exists, which has the
integral representation
\[
\mathcal{U}(\varPhi) = \int_0^1 \varPhi(p)\; d w(p),\quad \varPhi\in \Q_{\textup{b}},
\]
for some continuous nondecreasing function $w:(0,1]\to\R_{+}$. Setting $U(Z)=\mathcal{U}\big(F_Z^{-1}\big)$, we obtain~\eqref{quantile-int-var}.
\end{proof}

Another possibility is to consider the topology of uniform convergence,
induced by the norm
\[
\|Z\|_{\infty} = \sup_{\omega\in \varOmega}| Z(\omega)|.
\]
This means that we identify $\Z$ with the Banach space $B(\varOmega,\F)$ of bounded functions defined on $\varOmega$,
which can be obtained as uniform limits of simple functions. We  assume that the preorder $\prefeq$ is continuous
in this space.

\begin{theorem}
\label{t:rank-dependent-variables-Choquet}
Suppose $\Z=B(\varOmega,\F)$ and the probability space $(\varOmega,\F,P)$ is standard and atomless.
If the total preorder $\prefeq$ on $\Z$ is law invariant, continuous, monotonic,
and satisfies the dual independence axiom for random variables,
then  a nondecreasing function $w:[0,1]\to[0,1]$ exists, such that the functional
\begin{equation}
\label{schmeidler-var}
 U(Z)= -\int_{-\infty}^0 w\big(F_Z(\eta)\big) \;d\eta + \int_0^{\infty}\big[ 1 - w\big(F_Z(\eta)\big)\big]\;d\eta
\end{equation}
 is a numerical representation of $\,\prefeq$.
\end{theorem}
\begin{proof}
Recall that the preorder $\prefeq$ induces a preorder $\succeq$ on $\Q_{\textup{b}}$ defined in the proof of Theorem \ref{t:affine-dual-variables}.
It satisfies the monotonicity condition on $\Q_{\textup{b}}$, as in \eqref{equiv-orders}.
In order to use Theorem \ref{t:schmeidler}, we need to verify the continuity condition for $\succeq$.

Consider a uniformly convergent sequence of functions $\{\varPhi_n\}$ and a function $\varPsi$ in $\Q_{\textup{b}}$, such that
$\varPhi_n \succeq \varPsi$, $n=1,2\dots$, and let $\varPhi$ be the  uniform limit of  $\{\varPhi_n\}$, that is,
\[
\lim_{n\to\infty} \sup_{0\le p\le 1} |\varPhi_n(p)-\varPhi(p)|  =0.
\]
For a uniform random variable $Y$, we define $Z_n=\varPhi_n(Y)$, $Z=\varPhi(Y)$, and $V=\varPsi(Y)$. By \eqref{equiv-orders},
$Z_n\prefeq V$. Substituting the definitions of $Z_n$ and $Z$ and changing variables we obtain
\begin{align*}
\|Z_n-Z\|_{\infty} &= \sup_{\omega\in \varOmega} |\varPhi_n(Y(\omega))-\varPhi(Y(\omega))| \\
&= \sup_{0\le p \le 1} |\varPhi_n(p)-\varPhi(p)| \to 0, \quad \text{as}\quad n\to \infty.
\end{align*}
By the continuity of $\prefeq$ in $\Z$, we conclude that $Z\prefeq V$. By \eqref{equiv-orders}, $\varPhi\succeq\varPsi$. In a similar way we consider
the case when $\varPsi \succeq \varPhi_n$, $n=1,2\dots$ and we prove that $\varPsi \succeq \varPhi$.
Consequently, the preorder $\succeq$ is continuous in $\Q_{\textup{b}}$.

By Theorem \ref{t:schmeidler}, a numerical representation $\mathcal{U}(\cdot):\Q_{\textup{b}}\to\R$ of $\succeq$ exists, which has the
integral representation \eqref{schmeidler}
for some continuous nondecreasing function $w:(0,1]\to\R_{+}$. Setting $U(Z)=\mathcal{U}\big(F_Z^{-1}\big)$, we obtain \eqref{schmeidler-var}.
\end{proof}

Formula \eqref{schmeidler-var} is a special case of the Choquet integral of the variable $Z$ (see  \cite{Choquet}).
Clearly, if the assumptions of Theorem \ref{t:rank-dependent-variables-L1} are satisfied, so are the assumptions of Theorem \ref{t:rank-dependent-variables-Choquet}.
In this case, the representation \eqref{schmeidler-var} follows (by integration by parts and change of variables) from \eqref{quantile-int-var},
provided that the function $w(\cdot)$ in \eqref{quantile-int-var} is normalized so that  $w(1)=1$.

If we additionally assume that the preference relation $\prefeq$ is risk-averse in the sense of Definition \ref{d:risk-averse-vec},
we obtain the following corollary from Theorem \ref{t:dual-ra}.

\begin{corollary}
\label{c:dual-ra}
Suppose a total preorder $\prefeq$ on $\Z$ is continuous, monotonic, and satisfies the dual independence axiom for random vectors.
Then it is risk-averse if and only if it has the  numerical representation  \eqref{schmeidler-var} with
a nondecreasing and concave function $w:[0,1]\to[0,1]$ such that $w(0)=0$ and $w(1)=1$.
\end{corollary}

Similarly to the case of preferences among quantile functions, we need here the full
Definition \ref{d:risk-averse-vec}. This is in contrast to the expected utility theory
when the preference $\Eb[Z] \prefeq Z$ was sufficient (see Remark \ref{r:E-insuf}).


\bibliographystyle{plain}
\bibliography{Risk-Biblio}

\section*{Appendix}
For convenience, we provide here two integral representation theorems for continuous linear functionals on spaces of signed measures.
They are consequences of Banach's theorem on weakly$^\star$ continuous functionals \cite[VIII.8,Thm. 8]{Banach:1932}.
\begin{theorem}
\label{t:Mdual}
A functional $U:\M(\Sc)\to\R$ is  continuous and linear if and only if there exists $f\in \C_{\textup{b}}(\Sc)$
such that
\begin{equation}
\label{Urep}
U(\mu) = \int_{\Sc} f(z)\, \mu(dz), \quad \forall\,\mu\in \M(\Sc).
\end{equation}
\end{theorem}
\begin{proof}
Consider a compact set $K\subset \Sc$, and the space
\[
\M_K = \{\mu\in \M(\Sc): \text{supp}(\mu) \subseteq K\}.
\]
Every continuous linear functional on $\M(\Sc)$ is also a continuous linear  functional on $\M(K)$. The space $\M(K)$ can be identified
with the space of continuous linear functionals on $\C(K)$, the space of continuous functions on $K$. The topology
of weak convergence of measures in $\M(K)$ is exactly the weak$^\star$ topology on $[\C(K)]^*$. By Banach's theorem, every weakly$^\star$ continuous
functional $U(\cdot)$ on the dual space has the form
\begin{equation}
\label{UrepK1}
U(\mu) =\langle f_K,\mu\rangle =  \int_{K} f_K(z)\, \mu(dz),\quad \forall\,\mu\in \M(K),
\end{equation}
where $f_K\in \C(K)$.

Define  $f:\Sc\to \R$ as $f(z) = f_{\{z\}}(z)$.
If $z\in K$, then  $\M(\{z\})\subseteq \M(K)$. From \eqref{UrepK1} we conclude that $f(z)=f_K(z)$.
Consequently, \eqref{UrepK1} can be rewritten as follows:
\begin{equation}
\label{UrepK2}
U(\mu) = \int_{\Sc} f(z)\, \mu(dz), \quad \forall\,\mu\in \M(K),\ \forall\, K \subset \Sc.
\end{equation}

Observe that $f(z)=U(\delta_z)$. If $z_n\to z$, as $n\to\infty$, then
$\delta_{z_n}\wto \delta_{z}$. Owing to the continuity of $U(\cdot)$, we have $f(z_n)= U(\delta_{z_n})\to U(\delta_z)=f(z)$,
which  implies the continuity of $f(\cdot)$ on $\Sc$.

We shall prove that $f(\cdot)$ is bounded. Suppose the opposite,
that for every $n\ge 1$ we can find $z_n\in\Sc$ with $f(z_n)\ge n$. Consider the sequence of measures
$\mu_n =  \delta_{z_n}/\sqrt{n},\quad n=1,2,\dots$. On the one hand, $\mu_n\wto 0$ and thus
$U(\mu_n) \to U(0)$, when $n\to\infty$. On the other hand,
$U(\mu_n) =  {f(z_n)}/{\sqrt{n}} \to \infty$, {as} $n\to\infty$,
which is a contradiction. Consequently, $f\in \C_{\textup{b}}(\Sc)$.

It remains to prove that representation \eqref{UrepK2} holds true for every $\mu\in\M(\Sc)$. Since the space $\Sc$ is Polish,  every $\mu\in\M(\Sc)$
is tight, that is, for every $n=1,2,\dots$, there exists a compact set $K_n$ such that $|\mu|(\Sc\setminus K_n) < 1/n$.
Define the sequence of measures
$\mu_n$, $n=1,2,\dots$, as follows: $\mu_n(A) = \mu(A\cap K_n)$, for all $A\in \B$. By the definition of weak convergence,
$\mu_n\wto\mu$. Each $\mu_n\in \M_{K_n}$ and thus we can use \eqref{UrepK2} and the
continuity of $U(\cdot)$ to write
\[
U(\mu) = \lim_{n\to\infty} U(\mu_n) = \lim_{n\to\infty} \int_{\Sc} f(z)\, \mu_n(dz) = \int_{\Sc} f(z)\, \mu(dz).
\]
The last equation follows from the fact that $f\in \C_{\textup{b}}(\Sc)$ and $\mu_n\wto \mu$.
\end{proof}

\begin{theorem}
\label{t:Mdualpsi}
A functional $U:\M^{\psi}(\Sc)\to\R$ is  continuous and linear if and only if there exists $f\in \C_{\textup{b}}^{\psi}(\Sc)$
such that
\begin{equation}
\label{Ureppsi}
U(\mu) = \int_{\Sc} f(z)\, \mu(dz), \quad \forall\,\mu\in \M^\psi(\Sc).
\end{equation}
\end{theorem}
\begin{proof}
Every $\mu \in \M^\psi(\Sc)$ can be associated with a unique $\nu \in \M(\Sc)$, such that $\frac{d\nu}{d\mu} = \psi$.
The mapping $L:\M^\psi(\Sc)\to \M(\Sc)$ defined in this way is linear, continuous, and invertible.
Therefore, each linear continuous functional $U:\M^\psi(\Sc)\to\R$ corresponds to a linear continuous functional $U_0: \M(\Sc)\to\R$ as follows:
$U_0(\nu) = U(L^{-1}\nu)$,
and \emph{vice versa}: for every linear continuous functional $U_0: \M(\Sc)\to\R$ we have a corresponding $U:\M^\psi(\Sc)\to\R$ defined as
$U(\mu) = U_0(L\mu)$.

By Theorem \ref{t:Mdual}, there exists $f_0\in \C_{\textup{b}}(\Sc)$, such that
\[
U_0(\nu) = \int_{\Sc} f_0(z)\, \nu(dz), \quad \forall\,\nu\in \M(\Sc).
\]
Thus, for all $\mu\in \M^\psi(\Sc)$ we have
\[
U(\mu) = U_0(L\mu) = \int_{\Sc} f_0(z)\psi(z)\, \mu(dz).
\]
It follows that the representation \eqref{Ureppsi} is true with function $f = f_0 \psi$, which is an element of $\C_{\textup{b}}^\psi(\Sc)$.
The converse implication is evident.
\end{proof}

\end{document}